\documentclass[10pt,b5paper]{article}

\usepackage[T1]{fontenc}
\usepackage[utf8]{inputenc}
\usepackage[english]{babel}
\usepackage{amsmath,amssymb,amsthm}
\usepackage{lmodern}
\usepackage{geometry}
\usepackage{url}
\usepackage{titlesec}
\titleformat{\paragraph}[runin]{\normalfont\itshape}{}{0pt}{}[.]
\titlespacing*{\paragraph}{0pt}{0.7em}{0.5em}

\newtheorem{theorem}{Theorem}[section]
\newtheorem{lemma}[theorem]{Lemma}
\newtheorem{proposition}[theorem]{Proposition}
\newtheorem{corollary}[theorem]{Corollary}
\theoremstyle{definition}
\newtheorem{definition}[theorem]{Definition}
\theoremstyle{remark}
\newtheorem{remark}[theorem]{Remark}
\numberwithin{equation}{section}

\DeclareMathOperator{\E}{E}
\DeclareMathOperator{\Var}{Var}
\DeclareMathOperator{\Cov}{Cov}
\DeclareMathOperator{\Bin}{Bin}
\DeclareMathOperator{\Mult}{Mult}
\DeclareMathOperator{\sgn}{sgn}
\newcommand{\Z}{\mathbb{Z}}

\title{Discrete total variation under the multinomial,\\ Part~3:\\
expectation under a general distribution}
\author{
Nikola Bani\'c\\
Gideon Brothers, 10000 Zagreb, Croatia\\
\texttt{nbanic@gmail.com}
\and
Neven Elezovi\'c\\
Department of Applied Mathematics,\\
Faculty of Electrical Engineering and Computing,\\
University of Zagreb, 10000 Zagreb, Croatia\\
\texttt{neven.elezovic@fer.hr}
}
\date{\today}

\begin{document}
\maketitle

\begin{abstract}
For multinomial counts $X\sim\Mult(N;\mathbf p)$ on $n$ ordered bins and a reference law
$\mathbf r$ we study the expectation of the reference-centered discrete total variation
$d_{\mathbf r}(X)=\sum_{j=1}^{n-1}|(X_{j+1}-Nr_{j+1})-(X_j-Nr_j)|$, the natural extension of the
comb-test statistic to a non-uniform null. For a non-uniform null the raw total variation of the
counts carries a deterministic staircase of order $N$ that no centering of the statistic removes;
the counts are therefore centered at the reference law, and under the null
($\mathbf p=\mathbf r$) the statistic is the matched $d_{\mathbf p}$. Exact finite-sample
expectations are obtained by conditioning on the total mass in each adjacent pair, which reduces
every edge to a translated binomial absolute-deviation problem. For fixed $n\ge3$ and a sampling
law bounded away from the boundary of the simplex, the matched statistic admits a three-term
asymptotic expansion in powers of $1/N$ whose coefficients are explicit continuous periodic
functions of $N(p_{j+1}-p_j)$ (a generalized, non-Poincar\'e asymptotic expansion), obtained by a
weighted local Edgeworth estimate and a corner Euler--Maclaurin summation; all half-integer orders
vanish by parity. Each correction separates into a smooth Edgeworth part and an oscillatory lattice
part, the latter an Edgeworth perturbation of one formal Euler--Maclaurin series of a Gaussian
lattice sum. Under a sampling law $\mathbf p\neq\mathbf r$ the leading expectation is the
folded-normal profile implied by the central limit theorem; under local alternatives we identify
the $O(1)$ correction (a skewness term, plus a variance-gradient term for Pitman alternatives), and
fixed alternatives give a law-of-large-numbers limit, the
$\ell_1$ distance between the adjacent-gradient profiles of $\mathbf p$ and $\mathbf r$, on which
threshold tests are consistent.
\end{abstract}

\medskip\noindent{\sloppy\textbf{Keywords:} discrete total variation; multinomial distribution;
goodness-of-fit; folded normal distribution; Edgeworth expansion;
local limit theorem; Euler--Maclaurin summation; lattice distribution; Bernoulli polynomials.\par}


\section{Introduction}

	The two preceding papers in this series compute the expectation~\cite{banic_elezovic_p1}
	and the variance~\cite{banic_elezovic_p2} of the (linear) discrete total variation
	$d=d_{\mathbf u}=\sum_{i=1}^{n-1}|X_{i+1}-X_i|$ --- the histogram smoothness score
	of~\cite{banic_elezovic_tvor} and the statistic of the comb test~\cite{banic_elezovic_ct} ---
	under the \emph{uniform} null $\mathbf u=(1/n,\dots,1/n)$, the design case of the statistic.
	A practitioner, however, may wish to test a histogram against a \emph{general} reference
	law --- a known but non-uniform null distribution, listed as future work
	in~\cite{banic_elezovic_ct}. The present paper carries the first moment to that setting.

	When the null is not uniform the expected histogram is itself non-flat, and the raw total
	variation of the counts, $\|X\|_V=d_{\mathbf u}(X)$, carries a deterministic staircase of order $N$: by the law of
	large numbers $\|X\|_V/N\to\sum_j|r_{j+1}-r_j|$, which swamps the $O(\sqrt N)$ fluctuation that a
	test must see. (Here ``total variation'' is the variation of the histogram along the bin order, not
	the total-variation distance $\tfrac12\sum_i|p_i-q_i|$ between two laws.) The expected shape must
	be subtracted from the \emph{counts}, not from the statistic: subtracting a constant from
	$\|X\|_V$ leaves a signed sum (see the paragraph \emph{The uncentered statistic} after the proof of Theorem~\ref{thm:E-dXY}). The natural object
	is therefore the reference-centered statistic
	\[
		d_{\mathbf r}(X)=\sum_{j=1}^{n-1}
		|X_{j+1}-X_j-N(r_{j+1}-r_j)|.
	\]
	Under the null the sampling law is the reference law; writing this common law as
	$\mathbf p$, this is the matched statistic $d_{\mathbf p}$, the main object of
	\S\S\ref{sec:E-general}--\ref{sec:E-D1-detailed}. The same statistic $d_{\mathbf r}$ may also
	be evaluated under a different sampling law $\mathbf p$; then each mismatched edge
	carries a deterministic shift, and away from the matched regime its leading
	expectation is governed by an elementary folded-normal profile.

	The quantitative results are the following. The exact finite-$N$ computation,
	derived in \S\ref{sec:E-general} through
	Theorem~\ref{thm:F2-translated}, is
	per-edge: after conditioning on $K_j=X_j+X_{j+1}$, the conditional law of
	$X_{j+1}$ is binomial and the inner expectation is a translated binomial mean
	absolute deviation. This gives an exact finite mixture for $\E(d_{\mathbf p})$, collected in
	Proposition~\ref{prop:E-TVC-exact}, and, with a
	different translation, for $\E_{\mathbf p}(d_{\mathbf r})$
	(Proposition~\ref{prop:dXY-exact}). The closest classical analogue is the expected $\ell_1$
	distance $\sum_i|X_i-Np_i|$, a sum of binomial mean absolute deviations because each $X_i$ is
	binomial~\cite{beirlant1994}; an edge difference $D_j$ is not binomial, which is why the pair
	conditioning is needed. As $N\to\infty$ at fixed
	$n$ and away from boundary probabilities, the matched expectation has the expansion of Theorem~\ref{thm:E-TVC}, with per-edge input
	Lemma~\ref{lem:E-D1-asymp-general},
	\[
		\E(d_{\mathbf p})
		= \sqrt{\frac{2N}{\pi}}\,
		\sum_{j=1}^{n-1}s_j
		\left[1+\frac{c_j(N)}{N}+\frac{c_j^{(2)}(N)}{N^2}\right]+O(N^{-5/2}),
	\]
	where
	\begin{gather*}
		s_j^2=(p_j+p_{j+1})-(p_{j+1}-p_j)^2,\\
		c_j(N)=\frac{\lambda_{3,j}^2-\lambda_{4,j}}{24}
		-\frac{B_2(\{N(p_{j+1}-p_j)\})}{2s_j^2}.
	\end{gather*}
	Here $\lambda_{r,j}$ are the standardized cumulants of the one-trial increment of edge $j$,
	$B_2(t)=t^2-t+\tfrac16$, and $\{\cdot\}$ is the fractional part. Only $s_j$ and the first
	correction $c_j(N)$ are written out here; the longer coefficient $c_j^{(2)}(N)$ is stated in
	Lemma~\ref{lem:E-D1-asymp-general}. The first part of $c_j(N)$ is the smooth Edgeworth
	correction; the second is the lattice correction caused by the position of the kink relative
	to the integer lattice. Because $c_j(N),c_j^{(2)}(N)$ depend on $N$ only through
	$\{N(p_{j+1}-p_j)\}$, and continuously so, the expansion is a generalized (non-Poincar\'e)
	asymptotic expansion: bounded oscillating coefficient functions with a remainder uniform in
	the lattice position (Remark~\ref{rem:expansion-sense}). Each coefficient separates into a
	smooth Edgeworth part and an oscillatory lattice part; the $\lambda$-independent lattice parts
	of all orders are the terms of one formal (divergent) Euler--Maclaurin series of a Gaussian
	lattice sum, of which the actual coefficients are Edgeworth perturbations
	(Remark~\ref{rem:gaussian-backbone}). For $d_{\mathbf r}$ under a sampling law
	$\mathbf p$, treated in \S\ref{sec:E-dXY}, with gradient mismatch
	$\delta_i=(p_{i+1}-p_i)-(r_{i+1}-r_i)$, the leading per-edge term is the
	folded-normal expression $\sigma_iG(N\delta_i/\sigma_i)$
	(Theorem~\ref{thm:E-dXY}). The sub-leading correction under local alternatives is
	isolated in Theorem~\ref{thm:dXY-subleading}, and $d_{\mathbf r}/N$ converges to
	$\sum_i|\delta_i|$ for fixed alternatives in Corollary~\ref{cor:dXY-GC}.

	The paper develops this in four blocks after the multinomial preliminaries.  First,
	\S\ref{sec:E-general} defines the matched statistic and derives the exact finite-$N$
	per-edge formulas, ending with the mixture-conditioning reduction from a general edge
	to a translated binomial problem.  Next, \S\ref{sec:E-TVC-asymp} proves the
	asymptotic expectation of $d_{\mathbf p}$, with the statement in Theorem~\ref{thm:E-TVC}.  The technical proof of the per-edge expansion,
	Lemma~\ref{lem:E-D1-asymp-general}, is given separately in
	\S\ref{sec:E-D1-detailed}, since it combines a weighted local limit theorem and a corner
	Euler--Maclaurin summation of the lattice kink.  The final block, \S\ref{sec:E-dXY},
	treats $d_{\mathbf r}$ under a possibly different sampling law $\mathbf p$:
	Proposition~\ref{prop:dXY-exact} gives its exact finite-$N$ edge expression,
	Theorem~\ref{thm:E-dXY} gives the
	folded-normal leading term, \S\ref{ssec:dXY-n2} records what the companion
	papers~\cite{elezovic_mad,elezovic_folded} contribute in the excluded case $n=2$ --- where the
	statistic collapses to a binomial absolute deviation about the reference mean,
	and the exponentially small correction becomes explicit ---
	Corollary~\ref{cor:dXY-mixed} separates the matched,
	local, and shifted regimes, and Corollaries~\ref{cor:dXY-GC}--\ref{cor:dXY-consistency}
	record the law-of-large-numbers and consistency consequences.  The
	multinomial edge-variable toolbox on which all of this rests --- marginals, the two-stage
	binomial factorization, and the variance of a single edge --- is recalled in
	\S\ref{sec:multinomial-basics}.

	These results form a series on the exact and asymptotic moment theory of discrete total
	variation under the multinomial model, organized by \emph{which} moment is computed and
	\emph{under which} null: the expectation~\cite{banic_elezovic_p1} and the
	variance~\cite{banic_elezovic_p2} under the uniform null; the expectation --- this paper ---
	and the variance~\cite{banic_elezovic_p4} under a general reference law; and, in a final
	part~\cite{banic_elezovic_p5}, the cyclic statistic $d^{c}$ (the wrap edge $|X_1-X_n|$ added),
	to which the exact linear--cyclic identity of~\cite[Lemma~3.6]{banic_elezovic_p2} transfers the
	uniform variance theory. The present paper establishes exact expectation formulas for
	$d_{\mathbf p}$ in the matched null case and for $d_{\mathbf r}$ under a possibly different
	sampling law (for every $n$, $N$ and $\mathbf p$), together with asymptotic expansions in the
	fixed-interior regime: the asymptotic expansions (Theorems~\ref{thm:E-TVC}, \ref{thm:E-dXY},
	\ref{thm:dXY-subleading} and Corollary~\ref{cor:dXY-mixed}) are proved for fixed $n\ge3$ and a
	sampling law in the interior of the simplex; the law of large numbers and the consistency
	statements (Corollaries~\ref{cor:dXY-GC}--\ref{cor:dXY-consistency}) hold for every $n\ge2$; the
	boundary classification of \S\ref{sec:E-TVC-asymp} and the comparison with $\chi^2$ in
	\S\ref{sec:E-dXY} are heuristic; and the tables are numerical checks against the exact formulas. No random second sample is
	introduced; $\mathbf r$ is only the prescribed reference law in the statistic.

\section{Multinomial distribution: basics}
	\label{sec:multinomial-basics}

	This section records the elementary multinomial facts used throughout the paper:
	marginal distributions, the two-stage factorization of a pair of bins, and the
	variance of one edge difference. The statements are standard, but they are
	included with proofs in order to fix notation and to make the later reductions to
	one-dimensional binomial sums unambiguous.

	Two structural consequences are used repeatedly. First, the difference between
	adjacent counts depends only on the two-bin marginal $(X_j,X_{j+1})$. Second,
	after conditioning on the total mass
	$K_j=X_j+X_{j+1}$ in that pair, the remaining randomness is binomial. These two
	facts are the basis for the exact per-edge formulas in \S\ref{sec:E-general} and
	for the asymptotic expansions that follow.

% =====================================================================
\subsection{Definition and basic properties}
	\label{ssec:mult-def}

	Let $X = (X_1, X_2, \ldots, X_n)$ be a random vector following the
	multinomial distribution
	\begin{equation}\label{eq:mult}
		X \sim \Mult(N;\, p_1, p_2, \ldots, p_n), \qquad
		\sum_{i=1}^n p_i = 1,\quad p_i \ge 0.
	\end{equation}
	We interpret the components $X_i$ as the number of observations falling into the $i$-th bin
	in $N$ independent trials with probabilities $p_i$. The associated
	probability mass function is
	\begin{equation}\label{eq:mult-pmf}
		\Pr(X_1 = k_1, \ldots, X_n = k_n) =
		\frac{N!}{k_1!\,k_2!\cdots k_n!}\,p_1^{k_1}p_2^{k_2}\cdots p_n^{k_n},
		\quad \sum_i k_i = N,
	\end{equation}
	and $0$ otherwise. Here and below $0^0:=1$, and the $k_i$ range over nonnegative integers.

	\paragraph{Marginal distribution of a single bin} The marginal distribution of a
	single bin is binomial,
	\begin{equation}\label{eq:mult-marg-1}
		X_i \sim \Bin(N, p_i),
	\end{equation}
	so $\E[X_i] = N p_i$ and $\Var(X_i) = N p_i(1 - p_i)$.

	\paragraph{Covariance structure} The standard computation (via the representation
	$X_i = \sum_{k=1}^N \mathbf 1[Z_k = i]$ with $Z_k$ iid categorical on
	$\{1, \ldots, n\}$ with probabilities $p_i$) gives
	\begin{equation}\label{eq:mult-cov}
		\Cov(X_i, X_j) =
	\begin{cases}
		N p_i (1 - p_i), & i = j, \\
		- N p_i p_j, & i \ne j.
	\end{cases}
	\end{equation}

% =====================================================================
\subsection{Marginal distributions of subsets}
	\label{ssec:mult-marginals}

	The multinomial distribution has the following fundamental property: the marginal
	of any subset of variables is also multinomial, where the
	``remainder'' is treated as one additional bin.

\begin{proposition}[$k$-bin marginal]\label{prop:k-marginal}
	Let $S \subseteq \{1, \ldots, n\}$, $|S| = k$, and let
	$X_S^c := N - \sum_{i \in S} X_i$ be the number of observations outside $S$.
	The augmented vector consisting of the selected counts and the pooled remainder has
	distribution
	\begin{equation}\label{eq:k-marginal}
		\bigl((X_i)_{i \in S},X_S^c\bigr)
		\sim \Mult\bigl(N;\, (p_i)_{i \in S},\; 1 - \textstyle\sum_{i\in S} p_i\bigr),
	\end{equation}
	i.e.~a $(k{+}1)$-bin multinomial where the $(k{+}1)$-th bin is the ``combined remainder''
	with associated probability $1 - \sum_{i\in S} p_i$.
\end{proposition}

\begin{proof}
	Writing the corresponding multinomial probability in explicit form
	(\ref{eq:mult-pmf}) and summing over all $(k_i)_{i \notin S}$ with
	fixed $k_i$ for $i \in S$ (under the constraint $\sum_i k_i = N$):
	\[
		\Pr\bigl((X_i)_{i \in S} = (k_i)_{i \in S}\bigr)
		= \frac{N!}{\prod_{i\in S} k_i!}\,\prod_{i\in S}p_i^{k_i}
		\cdot \!\!\sum_{\substack{(k_i)_{i \notin S} \\ \sum_{i\notin S} k_i = m}}\!\!
		\frac{1}{\prod_{i\notin S} k_i!}\prod_{i\notin S} p_i^{k_i},
	\]
	where $m = N - \sum_{i\in S} k_i$. By the multinomial theorem the sum reduces to
	\[
		\frac{1}{m!}\bigl(\sum_{i\notin S} p_i\bigr)^m = 
		\frac{1}{m!}\bigl(1-\sum_{i\in S}p_i  \bigr)^m,
	\] 
	so
	\[
		\Pr\bigl((X_i)_{i \in S} = (k_i)_{i \in S}\bigr)
		= \frac{N!}{(\prod_{i\in S}k_i!)\,m!}\,
		\Bigl(\prod_{i\in S}p_i^{k_i}\Bigr)\bigl(1-{\textstyle\sum_{i\in S}}p_i\bigr)^m,
	\]
	which is exactly the probability mass function of the multinomial with parameters
	$\bigl(N;\,(p_i)_{i\in S},\,1-\sum_{i\in S}p_i\bigr)$.
\end{proof}

	We shall use the pair marginal repeatedly: $(X_j,X_{j+1})$, together with the
	pooled complement, has distribution
	$\Mult\bigl(N;p_j,p_{j+1},1-p_j-p_{j+1}\bigr)$. It is a three-category law
	although only two bins are selected.

% =====================================================================
\subsection{Edge variables}
	\label{ssec:edge-vars}

	The papers in this series are organized around the same local objects: differences
	between neighboring bins. Part~1 and Part~2 treat the uniform reference law, for
	which no centering is visible in these differences; the present paper and Part~4
	use the same edge variables with the non-uniform centering shown below.

\begin{definition}[Edge variables]\label{def:Dj}
	For a multinomial vector $X = (X_1, \ldots, X_n)$ we introduce the \emph{edge variables}
	\begin{equation}\label{eq:Dj-def}
		D_j := X_{j+1} - X_j, \qquad j = 1, \ldots, n-1,
	\end{equation}
	and their centered versions
	\begin{equation}\label{eq:tilde-Dj-def}
		\tilde D_j := D_j - \E[D_j] = D_j - N(p_{j+1} - p_j).
	\end{equation}
	If $p_j=p_{j+1}$ (in particular under the uniform law $p_i\equiv1/n$), then $\E[D_j]=0$
	and $\tilde D_j=D_j$.
\end{definition}

% =====================================================================
\subsection{Factorization of the 2-bin marginal}
	\label{ssec:2bin-factorization}

	If one technique deserves to be called the engine of this paper, it is the
	one in this subsection. The difficulty with the total variation is that the
	summands $|X_{j+1}-X_j|$ are entangled with the global constraint
	$\sum_i X_i = N$. The factorization below cuts that knot: it shows that, as far
	as a single edge is concerned, the entire $n$-bin multinomial reduces to a
	\emph{two-stage} binomial experiment --- first decide how many observations land
	in the pair $\{j, j+1\}$ at all, then split them between the two. Once the
	problem is phrased this way, the absolute value $|X_{j+1}-X_j|$ becomes the
	displacement of an ordinary random walk, making binomial identities directly
	available.

	Concretely, let
	\begin{equation}\label{eq:Kj-def}
		K_j := X_j + X_{j+1},
	\end{equation}
	the total of the pair $(X_j, X_{j+1}) \sim
	\Mult(N; p_j, p_{j+1}, 1-p_j-p_{j+1})$.

\begin{proposition}[Factorization of the 2-bin marginal]\label{prop:2bin-factor}
	Assume $p_j+p_{j+1}>0$. Then $K_j\sim\Bin(N,p_j+p_{j+1})$ and, for every $k$ with
	$\Pr(K_j=k)>0$ (all $0\le k\le N$ if $p_j+p_{j+1}<1$, only $k=N$ if $p_j+p_{j+1}=1$),
	\begin{equation}\label{eq:Xj1-cond}
		X_{j+1} \mid K_j = k \;\sim\; \Bin(k, q_j),
		\qquad q_j := \frac{p_{j+1}}{p_j + p_{j+1}}.
	\end{equation}
	Conditional on $K_j=k$, the complement $X_j=k-X_{j+1}$ is
	$\Bin(k,1-q_j)$.
\end{proposition}

\begin{proof}
	$K_j$ counts how many of the $N$ observations fall into the union of the bins $\{j, j+1\}$
	(success probability $p_j + p_{j+1}$ per independent trial), so it is
	$\Bin(N, p_j+p_{j+1})$.

	For the conditional distribution, by Proposition~\ref{prop:k-marginal} and (\ref{eq:mult-pmf}):
	\begin{align*}
		\Pr(X_{j+1} = m \mid K_j = k)
		&= \frac{\Pr(X_j = k-m, X_{j+1} = m)}{\Pr(K_j = k)} \\
		&= \frac{\binom{N}{k-m, m, N-k}\,p_j^{k-m}p_{j+1}^m (1-p_j-p_{j+1})^{N-k}}
		{\binom{N}{k}(p_j+p_{j+1})^k (1-p_j-p_{j+1})^{N-k}} \\
		&= \binom{k}{m}\,q_j^m (1-q_j)^{k-m},
	\end{align*}
	which is the probability mass function of $\Bin(k, q_j)$.
\end{proof}

	\paragraph{Implication for the differences $D_j = X_{j+1} - X_j$} Conditional on $K_j = k$:
	$D_j = 2X_{j+1} - k$, so $D_j \mid K_j = k$ has a shifted binomial distribution.
	When $p_j=p_{j+1}$ (in particular under the uniform law), $q_j=1/2$, so
	$X_{j+1} \mid K_j \sim \Bin(k, 1/2)$, which gives $D_j \mid K_j = k$ as
	the position of a symmetric $\pm 1$ random walk of $k$ steps. This
	identification is central to the derivations that follow.

% =====================================================================
\subsection{Variance of one edge}
	\label{ssec:edge-vars-var}

	The one second-order fact needed in the expectation asymptotics is the variance of
	a single edge. From this point on the absolute values enter, and with them the
	nonstandard part of the analysis.

\begin{lemma}[Variance of an edge variable]\label{lem:cov-Dj-multinomial}
	Under the multinomial distribution $\Mult(N; p_1, \ldots, p_n)$, each edge variable
	has variance
	\begin{equation}\label{eq:Var-Dj}
		\Var(D_j) = N\bigl[(p_j + p_{j+1}) - (p_j - p_{j+1})^2\bigr],
	\end{equation}
	unchanged by centering: $\Var(\tilde D_j) = \Var(D_j)$.
\end{lemma}

\begin{proof}
	By bilinearity, $\Var(D_j) = \Var(X_{j+1}) + \Var(X_j) - 2\Cov(X_j, X_{j+1})$.
	Substituting the multinomial moments \eqref{eq:mult-cov},
	\begin{align*}
		\Var(D_j) &= N p_{j+1}(1-p_{j+1}) + N p_j(1-p_j) + 2N p_j p_{j+1} \\
		&= N\bigl[(p_j+p_{j+1}) - (p_j-p_{j+1})^2\bigr].
	\end{align*}
	Centering shifts the mean only, so $\Var(\tilde D_j) = \Var(D_j)$.
\end{proof}

	For a general law the second moment
	$\E(X_{j+1}-X_j)^2=N^2(p_{j+1}-p_j)^2+N\bigl[(p_j+p_{j+1})-(p_j-p_{j+1})^2\bigr]$ appears
	in~\cite[Eq.~(11)]{banic_elezovic_tvor}; the bracket is \eqref{eq:Var-Dj}.

\section{The matched statistic under a general reference law}
	\label{sec:E-general}

	The discrete total variation of a histogram was introduced as a smoothness score
	in~\cite{banic_elezovic_tvor} and turned into a uniformity test in~\cite{banic_elezovic_ct}.
	For testing hypotheses of the type ``$\mathbf X\sim\Mult(N; \mathbf r)$'' with a non-uniform
	$\mathbf r$ we center the \emph{counts} at the reference law before taking the discrete total
	variation $\|y\|_V:=\sum_{j=1}^{n-1}|y_{j+1}-y_j|$; the uncentered statistic $\|X\|_V$
	of~\cite{banic_elezovic_tvor,banic_elezovic_ct} is the case $\mathbf r=\mathbf u$. For a
	probability vector $\mathbf r=(r_1,\ldots,r_n)$, in the series notation
	$d_{\mathbf r}(X)=\|X-N\mathbf r\|_V$,
	\begin{equation}\label{eq:dr-def}
		d_{\mathbf r}(X):=\sum_{j=1}^{n-1}\bigl|(X_{j+1}-Nr_{j+1})-(X_j-Nr_j)\bigr|
		=\sum_{j=1}^{n-1}\bigl|D_j-N(r_{j+1}-r_j)\bigr|.
	\end{equation}
	The present section treats the matched, or null, case: the sampling and reference
	laws coincide, and we then abbreviate
	$d_{\mathbf p}:=d_{\mathbf r}|_{\mathbf r=\mathbf p}$: the statistic with reference $\mathbf r$,
	evaluated when the sampling law $\mathbf p$ equals it. In this case
	$d_{\mathbf p}$ removes the deterministic edge means and leaves the
	fluctuation-scale quantity whose expectation is computed below. The final part of
	the paper keeps the same statistic $d_{\mathbf r}$ but evaluates its expectation
	when the sampling law is a possibly different vector $\mathbf p$.

% =====================================================================
\subsection{Setup: the matched statistic}
	\label{ssec:tvc-setup}

\begin{definition}[Matched statistic $d_{\mathbf p}$]\label{def:tvc}
	For $X \sim \Mult(N; \mathbf p)$, with the same vector used as reference, define
	\begin{equation}\label{eq:tvc}
		d_{\mathbf p} := d_{\mathbf r}|_{\mathbf r=\mathbf p}
		= \sum_{j=1}^{n-1}|X_{j+1} - X_j - N(p_{j+1} - p_j)| = \sum_{j=1}^{n-1}|\tilde D_j|,
	\end{equation}
	where the $\tilde D_j$ are centered edge variables (\ref{eq:tilde-Dj-def}). Under
	the uniform distribution $p_{j+1}-p_j \equiv 0$, so $\tilde D_j \equiv D_j$ and
	$d_{\mathbf p} \equiv d$.
\end{definition}

	\paragraph{Motivation for centering} For non-uniform $\mathbf p$, every edge with
	$p_{j+1}\ne p_j$ has $\E[D_j]\ne0$, so $|D_j|$ actually accumulates the deterministic
	shift ($|D_j| \approx |\E D_j| + $ stochastic part for large $N$). The statistic
	$d_{\mathbf p}$,
	by isolating $\tilde D_j$, extracts only the stochastic component. Indeed
	$\Pr(\sgn D_j\ne\sgn\E D_j)\le e^{-N(p_{j+1}-p_j)^2/2}$ (Hoeffding~\cite{hoeffding1963}), so on an edge with
	$p_{j+1}\ne p_j$ one has $\E|D_j|=N|p_{j+1}-p_j|+O(e^{-cN})$: the fluctuation term that
	$d_{\mathbf p}$ retains is lost (this is $d_{\mathbf r}$ with $\mathbf r=\mathbf u$ in
	\S\ref{sec:E-dXY}).
	Practical consequences:
	\begin{itemize}
		\item Under $H_0$: the sampling law equals the reference law, here denoted
		$\mathbf p$. In the practitioner's statistic $d_{\mathbf p}$ the deterministic
		$O(N)$ trend is removed exactly. Its null law still depends on
		$\mathbf p$ --- its leading mean through the edge variances, the corrections through the
		positions of the kinks $N(p_{j+1}-p_j)$. Cross-edge correlations are not needed
		for the expectation; they enter the variance theory, treated in Part~4 of the
		series \cite{banic_elezovic_p4} (for the uniform null see
		\cite[Lemma~2.4 and Theorems~4.6, 5.3]{banic_elezovic_p2}).
		\item Under $H_1$: the statistic used for the test remains $d_{\mathbf r}$, with
		$\mathbf r$ fixed by the null, while the sampling law is some $\mathbf p\ne\mathbf r$.
		This is the situation analyzed in \S\ref{sec:E-dXY}; the deterministic gradient
		mismatch then contributes to the expectation.
	\end{itemize}

	\paragraph{Why mean-centering} Centering each edge at its \emph{mean} is a design
	choice, and two natural variants deserve mention. The \emph{studentized} statistic
	$\sum_j|\tilde D_j|/s_j$ (with $s_j$ the per-edge standard deviation scale of
	\S\ref{sec:E-TVC-asymp}) has null expectation, for interior $\mathbf p$,
	$\sqrt{2N/\pi}\,(n-1)\,(1+O(N^{-1}))$, whose \emph{leading term} is free of
	$\mathbf p$; the full null law still depends on $\mathbf p$, so this is not
	asymptotic pivotality.
	\emph{Median-centering} would minimize each $\E|D_j-c|$ and keep the threshold on
	the lattice, which would change the oscillatory correction met below (we do not
	determine what replaces it). We adopt
	mean-centering because it makes each summand the absolute value of an exactly
	centered variable (so no
	location bias enters the folded-normal analysis of \S\ref{sec:E-dXY}), because it
	is the exact null instance $d_{\mathbf p}$ of the practitioner's statistic
	\eqref{eq:dr-def}, and because it continues the uniform statistic of
	Parts~1--2~\cite{banic_elezovic_p1,banic_elezovic_p2}. The \emph{exact} machinery below is
	indifferent to the choice: the translated formula of
	Theorem~\ref{thm:F2-translated} admits an arbitrary deterministic shift and the
	outer mixture is generic, so either variant is covered at no extra cost. The
	\emph{asymptotic} corrections are not indifferent: they record the position of the
	threshold relative to the integer lattice, so moving the centering to a lattice
	median would change them.

% =====================================================================
\subsection{Exact computation of the two-bin expectation}
	\label{ssec:exact-2bin}

	The main building block for $\E(d_{\mathbf p})$ is the exact two-bin expectation at
	finite $N$. The reason for isolating it is simple but important. Once the total
	number of observations falling into a neighboring pair is fixed, all remaining
	randomness is just the allocation between the two bins. Thus every edge
	expectation in the full $n$-bin model is obtained by first solving this two-bin
	allocation problem and then averaging over the random pair-total. The two-bin
	formula is therefore not a side calculation: it is the kernel of the exact
	finite-$N$ computation.

	The version needed for $d_{\mathbf p}$ is the \emph{translated} one, because the centered
	edge subtracts the unconditional mean $N(p_{j+1}-p_j)$ and not the conditional
	mean after the pair-total has been fixed; the shift is therefore a general real
	number, in general not an integer. The identity is classical: the partial
	expectation $\E[(c-Y)^+]$ is the binomial \emph{stop-loss} formula of actuarial
	mathematics; the truncated first moment on which it rests is Todhunter's
	formula~\cite[pp.~182--183]{todhunter1865}, abstracted for arbitrary limits as Lemma~1 of
	Diaconis and Zabell~\cite[p.~288]{diaconis_zabell}, who also tell the history back to
	De~Moivre; and the absolute moment of a discrete law about an arbitrary constant was given in
	closed form by Katti~\cite{katti1960}. We state it in the form used here; the new element is
	the outer mixture over $K_j$ and its use for adjacent multinomial differences.

	The unshifted case $t=0$ needs no separate treatment here: it is the $t=0$
	instance of Theorem~\ref{thm:F2-translated} below, which holds for an arbitrary
	real shift. At $t=0$ the fold sits at $k/2$. For $d_{\mathbf p}$ this case means
	$p_j=p_{j+1}$, so $q_j=\tfrac12$, the fold is the conditional mean, and De~Moivre's
	one-term collapse of Part~1 applies; for $q\neq\tfrac12$ (the raw statistic, or
	$d_{\mathbf r}$ under $\mathbf p\neq\mathbf r$ on an edge with $r_j=r_{j+1}$) the fold
	is off the mean and a binomial distribution function survives --- exactly what the
	translated formula returns. At $n=2$ both regimes are classical binomial problems: the
	matched one is the binomial mean absolute deviation~\cite[Theorem~4.2]{elezovic_mad},
	and the unshifted one is the biased walk, a fold at the prescribed center
	$N/2$~\cite[\S4]{elezovic_folded}.

\begin{theorem}[Translated two-bin expectation]\label{thm:F2-translated}
	Let $Y\sim\Bin(k,q)$ for an integer $k\ge0$, $\bar q:=1-q$, and let
	$t\in\mathbb R$ be an arbitrary
	(generally non-integer) shift. With $c:=\tfrac{k+t}{2}$, $m:=\lfloor c\rfloor$,
	and the binomial CDF $B(\ell;\kappa,q):=\Pr\{\Bin(\kappa,q)\le\ell\}$
	(with $B(\ell;\kappa,q):=0$ for $\ell<0$ and
	$B(\ell;\kappa,q):=1$ for $\ell\ge\kappa$), where the term multiplied by
	$kq$ is interpreted as zero when $k=0$,
	\begin{equation}\label{eq:F2-translated}
		\E\,|2Y-k-t|
		= k(q-\bar q)-t + 4\bigl[\,c\,B(m;k,q) - kq\,B(m-1;k-1,q)\,\bigr].
	\end{equation}
\end{theorem}

\begin{proof}
	Since $2Y-k-t = 2(Y-c)$, $\E|2Y-k-t| = 2\,\E|Y-c|$. Writing
	$|Y-c| = (Y-c) + 2(c-Y)^+$ and taking expectations gives
	$\E|Y-c| = (kq-c) + 2\,\E[(c-Y)^+]$, hence
	$\E|2Y-k-t| = 2(kq-c) + 4\,\E[(c-Y)^+]$ with $2(kq-c) = k(q-\bar q)-t$. Only
	$y<c$ contribute to the partial expectation. If $c$ is an integer, the endpoint
	$y=c$ contributes zero, so in all cases the sum may harmlessly be written through
	$m=\lfloor c\rfloor$.  If $m<0$ the sum is empty, and if $m\ge k$ it is the full
	binomial sum; these two boundary cases are exactly the CDF conventions in the
	statement:
	\[
		\E[(c-Y)^+] = \sum_{y=0}^{\min(m,k)}(c-y)\binom{k}{y}q^y\bar q^{\,k-y}
		= c\,B(m;k,q) - \sum_{y=1}^{m} y\binom{k}{y}q^y\bar q^{\,k-y}.
	\]
	Here and below terms outside $0\le y\le k$ are understood as zero; the stated
	CDF conventions also cover the cases $m<0$ and $m\ge k$.
	Using $y\binom{k}{y}q^y\bar q^{\,k-y} = kq\binom{k-1}{y-1}q^{y-1}\bar q^{\,(k-1)-(y-1)}
	= kq\,\Pr\{\Bin(k-1,q)=y-1\}$, the last sum equals $kq\,B(m-1;k-1,q)$, which is
	\eqref{eq:F2-translated}.
\end{proof}

% =====================================================================
\subsubsection{Generalization to general $n$: reduction to the pair marginal and mixture-conditioning}
	\label{sssec:mixture-conditioning-general}

	Now we want to compute $\E|\tilde D_j|$ for $n$ bins. The key
	idea is that $|\tilde D_j| = |X_{j+1} - X_j - N(p_{j+1}-p_j)|$
	\emph{depends only on the random pair $(X_j, X_{j+1})$}, and not on the other
	$X_i$ for $i \notin \{j, j+1\}$. Hence, we can restrict ourselves to
	studying the marginal distribution of this pair.

	\paragraph{Step 1: reduction to the pair marginal} By
	Proposition~\ref{prop:k-marginal}, the marginal distribution of the pair $(X_j, X_{j+1})$
	under $\Mult(N; p_1, \ldots, p_n)$ is the three-category law carried by the two selected
	bins and their pooled complement,
	\begin{equation}\label{eq:3-bin-marginal}
		(X_j, X_{j+1}) \sim \Mult\bigl(N; p_j, p_{j+1}, 1 - p_j - p_{j+1}\bigr).
	\end{equation}
	The third ``bin value'' $1 - p_j - p_{j+1}$ represents the aggregate probability
	of all other bins; the corresponding variable $X_j' := N - X_j - X_{j+1}$ is not
	of interest for $|\tilde D_j|$.
	Indeed, each trial has exactly three relevant outcomes: it falls in bin $j$, in
	bin $j+1$, or in their complement. Aggregating all complementary outcomes before
	counting therefore gives the stated three-category multinomial law. This is why
	the remaining $n-2$ coordinates can be discarded without approximation.

	\paragraph{Step 2: factorization of the 3-bin Mult into binomial + conditional binomial}
	By Proposition~\ref{prop:2bin-factor}, $K_j := X_j + X_{j+1} \sim \Bin(N, p_j + p_{j+1})$
	and $X_{j+1} \mid K_j = k \sim \Bin(k, q_j)$ with $q_j := p_{j+1}/(p_j + p_{j+1})$.
	The first statement counts the trials that fall anywhere in the selected pair.
	Given that exactly $k$ trials do so, those $k$ trials are independently allocated
	to bin $j+1$ with the renormalized probability $q_j$. Algebraically, if
	$\tau_j:=p_j+p_{j+1}$, then
	\[
		\Pr(X_{j+1}=y\mid K_j=k)
		=\frac{\binom{k}{y}p_{j+1}^{\,y}p_j^{\,k-y}}{\tau_j^k}
		=\binom{k}{y}q_j^y(1-q_j)^{k-y}.
	\]
	This factorization assumes $\tau_j>0$. If $\tau_j=0$, both selected bin counts
	vanish almost surely and the centered edge contribution is identically zero,
	so no conditional parameter is needed.

	\paragraph{Step 3: computation of $\E|\tilde D_j|$ via the factorization} Conditional on
	$K_j = k$:
	\begin{itemize}
		\item $D_j = X_{j+1} - X_j = 2X_{j+1} - k$.
		\item $\tilde D_j \mid K_j = k = 2X_{j+1} - k - N(p_{j+1}-p_j)$.
	\end{itemize}
	Note that $\tilde D_j$ is centered around the \emph{unconditional} expectation
	$N(p_{j+1}-p_j)$, and not around the conditional $\E[D_j \mid K_j]$. So the computation goes in
	two layers. First, for fixed $k$, let us introduce the translation constant
	$t:=N(p_{j+1}-p_j)$ (which in general is neither the conditional mean $k(2q_j-1)$ nor a point
	of the lattice $k+2\Z$ of $2X_{j+1}-k$, so it is not a mere recentering). Then
	\[
		\E[|2X_{j+1} - k - t| \mid K_j = k]
		= \E_{Y \sim \Bin(k, q_j)}\,|2Y - k - t|,
	\]
	which is the \emph{translated} two-bin expectation of
	Theorem~\ref{thm:F2-translated}, with input $k$ (instead of $N$), parameter
	$q_j$ (instead of $q$), and the shift $t$. For $t=0$ (an edge with
	$p_j=p_{j+1}$) it reduces to the unshifted expectation $\E|2Y-k|$.
	Notice the distinction between the two means:
	\[
		\E[D_j\mid K_j=k]=k(2q_j-1),
		\qquad
		\E D_j=N(p_{j+1}-p_j)=N\tau_j(2q_j-1).
	\]
	For $q_j\neq\tfrac12$ they agree only when $k=N\tau_j$ (possible only if $N\tau_j$ is an
	integer), i.e.\ only at the average value of the random pair-total.  This is the reason a translated formula is unavoidable: after
	conditioning on $K_j=k$, the centered threshold is usually off the conditional
	center and also off the integer lattice.

	Second, we average over $K_j \sim \Bin(N, p_j + p_{j+1})$:
	\begin{equation}\label{eq:E-D-mixture}
		\E|\tilde D_j| = \sum_{k=0}^{N}\Pr(K_j = k)\cdot
		\E[|2X_{j+1} - k - N(p_{j+1}-p_j)| \mid K_j = k].
	\end{equation}
	This is only the law of total expectation applied after the exact two-stage
	factorization, so no approximation has entered.  Formula \eqref{eq:E-D-mixture}
	is exact for any finite $N$ and any non-uniform $\mathbf p$. For a general
	translation $t=N(p_{j+1}-p_j)$, the sign changes
	at $Y=(k+t)/2$, so the inner expectation is a finite binomial sum split at
	$m=\lfloor(k+t)/2\rfloor$. The translated closed form is given by
	Theorem~\ref{thm:F2-translated}. For numerical use the formula is stable in double
	precision (relative error below $10^{-13}$ against exact rational arithmetic for $N\le400$ on the
	edges we tested); see also the remark after Proposition~\ref{prop:E-TVC-exact}.

	\paragraph{Step 4: $\E(d_{\mathbf p})$ as a sum over edges} By the definition
	$d_{\mathbf p} = \sum_{j=1}^{n-1}|\tilde D_j|$, and the expectation is linear, so
	\begin{equation}\label{eq:E-TVC-exact}
		\E(d_{\mathbf p}) = \sum_{j=1}^{n-1}\E|\tilde D_j|.
	\end{equation}

	The four steps combine into the exact finite-$N$ formula, which we record as a
	proposition because it is one of the main results of the paper and the starting point
	of every numerical evaluation below.

\begin{proposition}[Exact finite-$N$ expectation of $d_{\mathbf p}$]\label{prop:E-TVC-exact}
	Let $X\sim\Mult(N;\mathbf p)$ and fix an edge $j$. Write
	$\tau_j:=p_j+p_{j+1}$ and, when $\tau_j>0$,
	\[
		q_j:=\frac{p_{j+1}}{\tau_j},\qquad
		t_j:=N(p_{j+1}-p_j),\qquad
		c_k:=\frac{k+t_j}{2},\qquad
		m_k:=\lfloor c_k\rfloor .
	\]
	Then, for every finite $N$ and every $\mathbf p$,
	\begin{equation}\label{eq:E-D-mixture-closed}
		\E|\tilde D_j|
		=\sum_{k=0}^{N}\binom Nk\,\tau_j^{\,k}(1-\tau_j)^{N-k}\,\Lambda_j(k),
	\end{equation}
	where the inner (conditional) expectation is
	\[
		\Lambda_j(k):=k(2q_j-1)-t_j
		+4\bigl[\,c_k\,B(m_k;k,q_j)-k\,q_j\,B(m_k-1;k-1,q_j)\,\bigr]
	\]
	(in particular $\Lambda_j(0)=|t_j|$: both bins are then empty, so $D_j=0$ and
	only the centering $t_j$ survives),
	with the binomial CDF $B(\cdot\,;\cdot,\cdot)$ and the conventions of
	Theorem~\ref{thm:F2-translated}, and
	\[
		\E(d_{\mathbf p})=\sum_{j=1}^{n-1}\E|\tilde D_j| .
	\]
	If $\tau_j=0$ the edge contributes $0$.
\end{proposition}

\begin{proof}
	If $\tau_j=0$ then $X_j=X_{j+1}=0$ almost surely and $\tilde D_j=0$. Otherwise
	Proposition~\ref{prop:2bin-factor} gives $K_j\sim\Bin(N,\tau_j)$ --- the outer weights in
	\eqref{eq:E-D-mixture-closed} --- and $X_{j+1}\mid K_j=k\sim\Bin(k,q_j)$. On $\{K_j=k\}$ one
	has $\tilde D_j=2X_{j+1}-k-t_j$, so the inner conditional expectation is
	$\E|2Y-k-t_j|$ with $Y\sim\Bin(k,q_j)$, which is $\Lambda_j(k)$
	by Theorem~\ref{thm:F2-translated} (there $q=q_j$,
	$\bar q=1-q_j$, hence $q-\bar q=2q_j-1$). The law of total expectation \eqref{eq:E-D-mixture} and the
	linearity \eqref{eq:E-TVC-exact} finish the proof. No approximation enters at any step.
\end{proof}

	Each $\Lambda_j(k)$ costs two evaluations of the binomial CDF, so
	\eqref{eq:E-D-mixture-closed} is a single sum of length $N+1$; in practice the outer
	binomial weights are negligible outside a window of width $O(\sqrt{N\tau_j(1-\tau_j)})$
	about $N\tau_j$, and truncating there is a numerical acceleration, not part of the
	identity.

\section{Asymptotic expectation of the matched statistic $d_{\mathbf p}$}
	\label{sec:E-TVC-asymp}

	This section derives the large-$N$ expansion of the matched statistic $d_{\mathbf p}$ at fixed
	$n$. The exact finite-$N$ formulas of \S\ref{sssec:mixture-conditioning-general}
	are useful for computation, but the asymptotic analysis is more direct at the
	level of a single edge. We write $D_j$ as a sum of independent three-point
	increments and apply a weighted local Edgeworth expansion to the corresponding
	lattice law. A corner Euler--Maclaurin summation then separates the continuous Gaussian contribution
	from the lattice correction caused by the position of the kink of
	$|x-N\mu_j|$ relative to the integer lattice. The same
	$\sqrt{2\sigma^2/\pi}\,[1+\text{Edgeworth}/N+\text{oscillatory Bernoulli}]$ mechanism
	underlies the mean absolute deviation of the classical discrete
	laws~\cite{elezovic_mad,elezovic_madfamily}; the per-edge expansion below is its
	multinomial adjacent-difference instance.

	After centering, the deterministic order-$N$ contribution is removed and the
	leading scale is $\sqrt N$. The leading coefficient is the usual
	$\sqrt{2/\pi}$ factor times the standard deviation of the edge. The first correction
	has two components: a smooth Edgeworth term, determined by the third and fourth
	cumulants of one increment, and an oscillatory lattice term involving
	$B_2(\{N\mu_j\})$. The next correction is stated at the same time in the lemma below;
	its derivation, including the cumulants up to order six and the higher corner terms of
	the Euler--Maclaurin expansion, is given in \S\ref{sec:E-D1-detailed}.
	We shall use the Bernoulli polynomials of degrees $2$--$4$,
	\[
		B_2(t)=t^2-t+\tfrac16,\qquad
		B_3(t)=t^3-\tfrac32t^2+\tfrac12t,\qquad
		B_4(t)=t^4-2t^3+t^2-\tfrac1{30}.
	\]

\begin{lemma}[Per-edge expansion]
	\label{lem:E-D1-asymp-general}
	Fix an edge $j$ with $0<p_j+p_{j+1}<1$ (one of $p_j,p_{j+1}$ may vanish), and set
	$\mu_j = p_{j+1} - p_j$, $\tau_j = p_j + p_{j+1}$, $s_j^2 = \tau_j - \mu_j^2$,
	$\sigma_j^2 = N s_j^2$. Then, as $N \to \infty$,
	\begin{equation}\label{eq:E-D1-asymp-general}
		\E|\tilde D_j| = \sqrt{\frac{2}{\pi}}\,\sigma_j
		\Bigl[\,1 + \frac{c_j(N)}{N} + \frac{c_j^{(2)}(N)}{N^2} + O(N^{-3})\Bigr],
	\end{equation}
	where, with the standardized cumulants $\lambda_r=\lambda_{r,j}:=\kappa_{r,j}/s_j^r$ of the per-trial
	increment (the edge index is suppressed) and the periodic Bernoulli polynomials ($\{\,\cdot\,\}$ the fractional part),
	\begin{align}
		c_j(N) &= \frac{\lambda_3^2 - \lambda_4}{24} - \frac{B_2(\{N\mu_j\})}{2\,s_j^2},
		\label{eq:coefficients_cj}\\
		c_j^{(2)}(N) &= \frac{24\lambda_6 - 120\lambda_3\lambda_5 - 75\lambda_4^2
		+ 350\lambda_3^2\lambda_4 - 175\lambda_3^4}{5760} + \ell_j(N),\label{eq:cj2}\\
		\ell_j(N) &= \frac{5\lambda_3^2 - 3\lambda_4}{48\,s_j^2}\,B_2(\{N\mu_j\})
		- \frac{\lambda_3}{6\,s_j^3}\,B_3(\{N\mu_j\})
		+ \frac{1}{8\,s_j^4}\,B_4(\{N\mu_j\}).\label{eq:ell-j}
	\end{align}
	The cumulants of the per-trial increment are
	\begin{gather}
		\kappa_3 = \mu_j(1 - 3\tau_j + 2\mu_j^2),\qquad
		\kappa_4 = \tau_j - 3\tau_j^2 - 4\mu_j^2 + 12\tau_j\mu_j^2 - 6\mu_j^4,\notag\\
		\kappa_5 = \mu_j\bigl(24\mu_j^4 - 60\mu_j^2\tau_j + 20\mu_j^2 + 30\tau_j^2 - 15\tau_j + 1\bigr),
		\label{eq:kappa5}\\
		\kappa_6 = -120\mu_j^6 + 360\mu_j^4\tau_j - 120\mu_j^4 - 270\mu_j^2\tau_j^2
		+ 150\mu_j^2\tau_j - 16\mu_j^2 + 30\tau_j^3 - 15\tau_j^2 + \tau_j.\notag
	\end{gather}
\end{lemma}

	The proof of Lemma~\ref{lem:E-D1-asymp-general} occupies the whole of
	\S\ref{sec:E-D1-detailed}. It starts from the weighted lattice sum
	\[
		\E|\tilde D_j|=\sum_{m\in\mathbb Z}|m-N\mu_j|\,\Pr(D_j=m),
	\]
	replaces the exact lattice mass by a weighted local Edgeworth approximation, and
	then applies a corner Euler--Maclaurin formula at the kink. The smooth part gives
	the Edgeworth corrections in $c_j(N)$ and $c_j^{(2)}(N)$, while the corner terms
	give the Bernoulli-polynomial lattice corrections.

	Through the order retained here (relative $N^{-3}$) the expansion proceeds in integer powers
	of $N^{-1}$: every half-integer order vanishes by parity (Lemma~\ref{lem:parity}), and the same
	parity argument applies at every order. In the smooth part the odd
	Edgeworth polynomials integrate to zero against the even weight $|z|\phi$; in the
	lattice part the corner term of index $k$, read against the Edgeworth profile of
	index $r$, sits at relative order $N^{-(k+r)/2}$ and carries an odd derivative of
	$\phi$ at its center whenever $k+r$ is odd. This is why the two displayed
	corrections are followed by $O(N^{-3})$ and not by $O(N^{-5/2})$; the same
	phenomenon makes the exceptional case $n=2$ an expansion in integer powers as
	well \cite[Theorem~4.2]{elezovic_mad}. The coefficients $c_j(N),c_j^{(2)}(N)$ are
	bounded but $N$-dependent through $\{N\mu_j\}$.

\begin{remark}[In what sense this is an asymptotic expansion]\label{rem:expansion-sense}
	The coefficients $c_j(N),c_j^{(2)}(N)$ depend on $N$ only through the fractional part
	$\{N\mu_j\}$, and only through the periodic Bernoulli functions $B_k(\{\cdot\})$, $k\ge2$, which
	are continuous; thus $c_j(N)=\Phi_{1,j}(N\mu_j)$ and $c_j^{(2)}(N)=\Phi_{2,j}(N\mu_j)$ for explicit
	continuous $1$-periodic functions $\Phi_{1,j},\Phi_{2,j}$, and \eqref{eq:E-D1-asymp-general} is an
	asymptotic expansion in the generalized sense of Erd\'elyi and Wyman with respect to the scale
	$\{N^{-k}\}$ --- bounded oscillating coefficient functions and a remainder uniform in the
	phase --- rather than in the Poincar\'e sense. The estimate is
	valid uniformly in the position of the kink relative to the lattice --- a uniformity the proof
	delivers (Lemma~\ref{lem:weighted-local-edgeworth} is uniform in $m$, Step~1 of
	\S\ref{sec:E-D1-detailed} is uniform in the grid shift, and the corner estimates do not depend
	on $N\mu_j$). The $N$-dependence follows a clean dichotomy: for rational $\mu_j=u/v$ the
	coefficients are periodic in $N$ with period dividing $v$; for irrational $\mu_j$ the sequence $\{N\mu_j\}$ equidistributes,
	the coefficients fill their range densely and never converge, and, since $\int_0^1B_k=0$, their
	Ces\`aro means are the smooth parts (for rational $\mu_j=u/v$ in lowest terms the period mean
	of $B_k(\{Nu/v\})$ is $v^{-k}B_k(0)$ instead, which is why the uniform case freezes at $-n/24$
	below). Oscillatory Bernoulli-polynomial corrections are a known feature of Edgeworth
	expansions for lattice distributions (Esseen~\cite{esseen1945}; for their absorption into
	Sheppard-corrected cumulants, Kolassa--McCullagh~\cite{kolassa_mccullagh1990}), and the kinked
	instance is already the binomial mean absolute deviation, whose first correction is
	Frame's~\cite{frame1945} and whose complete expansion is in~\cite{elezovic_mad} (other classical
	laws:~\cite{elezovic_madfamily}). What is specific here is the three-point increment of the edge
	difference, for which no collapse identity is available, together with its second-order
	coefficient $c_j^{(2)}(N)$, uniformly in the position of the kink $N\mu_j$ relative to the
	lattice.
\end{remark}

\begin{theorem}[Asymptotic $\E(d_{\mathbf p})$]\label{thm:E-TVC}
	Fix $n\ge3$, $p_{\min}>0$, and a probability vector
	$\mathbf p=(p_1,\ldots,p_n)$ with $\min_jp_j\ge p_{\min}$. Then, as $N \to \infty$,
	\begin{equation}\label{eq:E-TVC}
		\E(d_{\mathbf p}) = \sqrt{\frac{2N}{\pi}}\sum_{j=1}^{n-1} s_j
		\Bigl[\,1 + \frac{c_j(N)}{N} + \frac{c_j^{(2)}(N)}{N^2}\Bigr] + O(N^{-5/2}),
	\end{equation}
	with $c_j(N),c_j^{(2)}(N),s_j,\mu_j$ from Lemma~\ref{lem:E-D1-asymp-general}. The
	$O(N^{-5/2})$ is an \emph{absolute} remainder for the expectation, and its constant is
	uniform over the class
	$\{\mathbf p:\min_j p_j\ge p_{\min}\}$. In particular the
	leading term
	\begin{equation}\label{eq:E-TVC-leading}
		\E(d_{\mathbf p}) = \sqrt{\frac{2N}{\pi}}
		\sum_{j=1}^{n-1}\sqrt{(p_j + p_{j+1}) - (p_j - p_{j+1})^2} + O(N^{-1/2})
	\end{equation}
	holds uniformly over the same class.
\end{theorem}

\begin{proof}
	The centered $\tilde D_j = D_j - N(p_{j+1}-p_j)$ has mean zero and variance
	$\Var(\tilde D_j)=N[(p_j+p_{j+1})-(p_j-p_{j+1})^2]$ (Lemma~\ref{lem:cov-Dj-multinomial}). Since
	$n\ge3$ and every bin has positive mass, each edge satisfies $p_j+p_{j+1}<1$, so
	Lemma~\ref{lem:E-D1-asymp-general} applies edge by edge.  The lower bound
	$\min_jp_j\ge p_{\min}$ also gives
	$1-p_j-p_{j+1}=\sum_{\ell\ne j,j+1}p_\ell\ge(n-2)p_{\min}>0$, so each edge
	triple $(p_j,p_{j+1},1-p_j-p_{j+1})$ stays in a compact subset of the open
	three-simplex.  As $\sigma_j=\sqrt N\,s_j$, summing the
	per-edge expansion gives \eqref{eq:E-TVC}, the edgewise relative remainder $O(N^{-3})$ becoming an
	absolute $O(N^{-5/2})$ after multiplication by $\sigma_j=O(\sqrt N)$ over the fixed number of
	edges. Retaining only the leading part gives \eqref{eq:E-TVC-leading}, whose uniformity over
	$\{\mathbf p:\min_jp_j\ge p_{\min}\}$ is inherited from the compact-uniformity clause of
	Lemma~\ref{lem:weighted-local-edgeworth}: every edge triple
	$(a,b,w)=(p_j,p_{j+1},1-p_j-p_{j+1})$ lies in a compact subset of the open simplex, Steps~1--5
	of \S\ref{sec:E-D1-detailed} use only upper and lower bounds on $s_j^2$, and the maximum over the
	$n-1$ edgewise constants preserves the order.
\end{proof}

	The expansion also holds, with a constant depending on $\mathbf p$, for any fixed $\mathbf p$
	with $p_j+p_{j+1}<1$ on every edge, empty bins allowed; an edge with $p_j=p_{j+1}=0$
	contributes $0$ exactly (Proposition~\ref{prop:E-TVC-exact}) and is omitted from the sum in
	\eqref{eq:E-TVC}. Only the uniformity clause needs
	$\min_jp_j\ge p_{\min}$. The coefficients scale like inverse powers of $s_j^2\le p_j+p_{j+1}$,
	so the expansion is effectively one in $1/(Ns_j^2)$ --- roughly the reciprocal of the expected
	count in the pair --- and is numerically useful when $N(p_j+p_{j+1})\gg1$ on every edge (on
	the edge $(0.02,0.05)$, with $s_j^2=0.069$, the coefficient $c_j^{(2)}(N)$ oscillates in
	$[-2.7,\,0.9]$, against $[-0.03,\,0.015]$ on the edge $(0.1,0.3)$ below). The first-order asymptotics
	$\E(d_{\mathbf p})\sim\sqrt{2N/\pi}\sum_js_j$ (and, under a bounded standardized mismatch, the
	folded-normal leading term of Theorem~\ref{thm:dXY-subleading}) also
	follow from the multinomial central limit theorem and uniform integrability
	($\sup_N\E\tilde D_j^2/N<\infty$); the $O(N^{-1/2})$ remainder of \eqref{eq:E-TVC-leading} and
	the two correction terms, with their lattice structure, are the content of
	Theorem~\ref{thm:E-TVC}.

	\paragraph{Boundary regimes (heuristic, out of scope)} The assumption
	$p_j \ge p_{\min} > 0$ is not technical: for $p_j$ going to zero with $N$, we leave the central-limit-theorem
	domain on the affected coordinates. The following classification lies outside the
	$p_{\min}>0$ theorems and is stated without proof, as a guide to what is excluded. If both
	probabilities on an edge are of order $1/N$, the edge enters a Poisson/Skellam regime and
	contributes $O(1)$ rather than $O(\sqrt N)$ (a Skellam edge is elementary and confirms this);
	if only one probability vanishes while its neighbor stays bounded away from zero, the edge
	variance remains of order $N$ and the $\sqrt N$ scale persists. Thus boundary behavior must
	be classified edge by edge; the $p_{\min}$ condition in Theorem~\ref{thm:E-TVC} excludes all
	such mixed regimes.

	\paragraph{The excluded case $n=2$} The case $n=2$ was excluded from
	Theorem~\ref{thm:E-TVC} for a different reason: its single edge has
	$\tau_1=p_1+p_2=1$, the increment of $D_1=X_2-X_1$ takes only the values
	$\pm1$, the walk has span~$2$, and Lemma~\ref{lem:E-D1-asymp-general} does not
	apply. The obstruction is only apparent --- the matched statistic collapses to
	a binomial mean absolute deviation about its mean, $d_{\mathbf p}=2|X_1-Np|$: the exact
	De~Moivre form is classical~\cite{diaconis_zabell}, and the complete expansion is
	\cite[Theorem~4.2]{elezovic_mad} (its first coefficient is Frame's~\cite{frame1945}). We
	quote only what the present series uses: with $p:=p_1$, $q:=1-p$, $s_1^2=4pq$ and
	$h_N:=\lceil Np\rceil-Np$,
	\begin{equation}\label{eq:E-TVC-n2}
		\E(d_{\mathbf p})=\sqrt{\frac{2N}{\pi}}\,s_1
		\Bigl[\,1+\frac{\tilde c_1(N)}N+O(N^{-2})\Bigr],
		\qquad
		\tilde c_1(N)=\frac1{12}-\frac{B_2(h_N)}{2pq},
	\end{equation}
	an expansion in \emph{integer} powers of $N^{-1}$ whose coefficients oscillate
	through the lattice defect of the mean
	\cite[Theorem~4.2]{elezovic_mad} (see also \cite[\S2]{elezovic_folded}).

	It is worth recording that this is not a separate theory. The coefficients of
	\eqref{eq:E-TVC-n2} are \emph{exactly} those of
	Lemma~\ref{lem:E-D1-asymp-general}, read at $\tau_j=1$ on a lattice of span
	$2$. Indeed, putting $\tau=1$ in the cumulants of that lemma gives
	$\kappa_3=-2\mu s^2$ and $\kappa_4=2s^2(3\mu^2-1)$, so $\lambda_3=-2\mu/s$,
	$\lambda_4=2(3\mu^2-1)/s^2$, and the smooth coefficient becomes
	\[
		\frac{\lambda_3^2-\lambda_4}{24}
		=\frac{4\mu^2-6\mu^2+2}{24\,s^2}=\frac{2(1-\mu^2)}{24\,s^2}=\frac1{12},
	\]
	the smooth constant of \eqref{eq:E-TVC-n2}. The lattice term
	$-B_2(\{N\mu_j\})/(2s_j^2)$, on the other hand, must be re-derived on the
	coarse lattice. There it becomes
	\[
		-\frac{2^2B_2(h_N)}{8pq}=-\frac{B_2(h_N)}{2pq},
	\]
	again \eqref{eq:E-TVC-n2}. The rule is
	\begin{equation}\label{eq:span2-rule}
		B_k(\{N\mu_j\})\ \longmapsto\ 2^k\,B_k(h_N),
	\end{equation}
	and it must be read as a \emph{re-derivation}, not as an identity between
	Bernoulli polynomials: one has $\{N\mu_1\}=\{2h_N\}$ exactly (since
	$N\mu_1=N-2Np$), whereas $B_k(\{2h\})\ne2^kB_k(h)$ in general. What
	\eqref{eq:span2-rule} records is that the edge moves in steps of $2$, so a
	corner Euler--Maclaurin summation over a lattice of spacing $d$ divides the
	leading integral by $d$ and multiplies the $k$-th corner term by $d^{k-1}$ ---
	whence a factor $d^k=2^k$ \emph{relative} to the leading term --- and that the
	fractional position of the kink must be measured on the coarse lattice on which
	the kink actually sits, namely by $h_N=\lceil Np\rceil-Np$, the defect of $Np$
	on the span-one lattice of $X_1$. The exceptional case is therefore the same
	expansion with the parity obstruction absorbed into the span, which is also why
	it runs, like the general one, in integer powers of $N^{-1}$
	(Lemma~\ref{lem:parity}).

\begin{remark}[The Gaussian backbone]\label{rem:gaussian-backbone}
	The lattice corrections of \eqref{eq:coefficients_cj} and \eqref{eq:ell-j} share
	one structure. Setting $\lambda_r=0$ for all $r\ge3$ in \eqref{eq:edgeworth-pseudo} --- a
	Gaussian \emph{profile}; no lattice increment is Gaussian --- the odd
	derivatives in \eqref{eq:corner-EM} vanish and
	$f_N^{(2i-2)}(x_0) = (-1)^{i-1}(2i-3)!!\,\phi(0)/\sigma_j^{2i-1}$ (the derivative entering the $k=2i$ corner), so only the even $k=2i$
	survive; using $(2i-1)(2i-3)!!=(2i-1)!!$ (with $(-1)!!:=1$) and $(2i)!=2^i i!\,(2i-1)!!$, the
	lattice correction collapses to the single Euler--Maclaurin series of the Gaussian lattice sum
	$\sum_m|m-x_0|\,\sigma_j^{-1}\phi\bigl((m-x_0)/\sigma_j\bigr)$ (a theta-type sum)
	\begin{equation}\label{eq:theta-backbone}
		\sum_{i\ge1}\frac{(-1)^i}{2^i\,i!}\,
		\frac{B_{2i}(\{N\mu_j\})}{s_j^{2i}\,N^i}
		= -\frac{B_2(\{N\mu_j\})}{2s_j^2\,N}
		+ \frac{B_4(\{N\mu_j\})}{8s_j^4\,N^2} - \cdots,
	\end{equation}
	the folded standard Gaussian read on the lattice shifted by $\{N\mu_j\}$. The
	leading lattice term $-B_2/(2s_j^2)$ of $c_j$ is its $i=1$ member,
	the $B_4$-term of \eqref{eq:ell-j} its $i=2$ member, and the
	$\lambda_3,\lambda_4$ pieces of \eqref{eq:ell-j} are the first Edgeworth
	corrections to it. This is why the off-lattice kink of
	Remark~\ref{rem:expansion-sense} produces a whole family of oscillatory terms
	rather than one. Two qualifications: the collapse to the single series
	\eqref{eq:theta-backbone} concerns the $\lambda$-independent part of the coefficients, and the
	series is asymptotic, not convergent --- $\sup_t|B_{2i}(t)|\sim2(2i)!/(2\pi)^{2i}$ and
	$|B_{2i}(t)|$ is of factorial size for every $t$, so
	\eqref{eq:theta-backbone} diverges for every fixed $N$ and is to be read as a formal
	generating pattern for the lattice coefficients.
\end{remark}

	\paragraph{Check: the uniform reduction} The general per-edge coefficient specializes
	to the known uniform one. For $p_j=p_{j+1}=1/n$ with $n\ge3$: $\mu_j=0$,
	$\tau_j=s_j^2=2/n$, $\kappa_3=0$, and $\kappa_4=2/n-3(2/n)^2$, so
	$\lambda_3=0$, $\lambda_4=(n-6)/2$, and the smooth part of $c_j$ is
	$(\lambda_3^2-\lambda_4)/24=-(n-6)/48$. Since $\mu_j=0$ we have $\{N\mu_j\}=0$ and
	$B_2(0)=1/6$, so the lattice part is $-B_2(0)/(2s_j^2)=-(1/6)/(4/n)=-n/24$. Their sum is
	\[
		-\frac{n-6}{48}-\frac n{24}=-\frac{3n-6}{48}=-\frac{n-2}{16},
	\]
	the uniform per-edge coefficient $e_1$ of \cite[Corollary~5.6]{banic_elezovic_p1} (stated
	there for the cyclic statistic; under the uniform law all edges have the same law). Note
	that the lattice term does \emph{not} vanish in the uniform case; it freezes at the
	fixed value $-n/24$. What disappears under uniformity is its \emph{oscillation} with
	$N$, which is why the uniform theory of Part~1 never had to name it. At the next order, with
	$\lambda_6=\tfrac{n^2}4-\tfrac{15n}2+30$ and $B_4(0)=-\tfrac1{30}$, one finds
	$c_j^{(2)}=-(3n^2-12n-4)/512$, which is $e_2$ of \cite[Corollary~5.6]{banic_elezovic_p1} --- the
	strongest available independent check of the second-order coefficient.

	\paragraph{A concrete non-uniform edge}
	The edge $(p_j,p_{j+1})=(0.1,0.3)$ will serve as a running example. Here
	$\mu_j=0.2$, $\tau_j=0.4$, $s_j^2=0.36$, and
	\[
		\kappa_3=0.2\,(1-1.2+0.08)=-0.024,\qquad \kappa_4=-0.0576,
	\]
	so $\lambda_3=-\tfrac19$, $\lambda_4=-\tfrac49$, and the smooth coefficient is
	$(\lambda_3^2-\lambda_4)/24=37/1944\approx0.01903$. The full coefficient
	$c_j(N)=37/1944-B_2(\{N/5\})/0.72$ is periodic in $N$ with period $5$; at
	$N\equiv0\pmod5$ it equals $37/1944-(1/6)/0.72\approx-0.21245$. This displays the genuine
	$N$-dependence of Remark~\ref{rem:expansion-sense} in its simplest form --- here as exact
	periodicity, because $\mu_j$ is rational.

	\paragraph{Numerical validation} Table~\ref{tab:E-TVC-leading} compares successive
	truncations of \eqref{eq:E-TVC} --- the leading term \eqref{eq:E-TVC-leading}, the
	$O(N^{-1})$ correction, and the full displayed expansion including
	$c_j^{(2)}(N)/N^2$ --- with exact values computed from the finite mixture, for two
	non-uniform reference laws on $n=10$ bins: the \emph{comb}~\cite{banic_elezovic_ct}
	\[
		p_k=\frac1n+\frac{(-1)^k}{2n}\qquad(k=1,\dots,n),
	\]
	alternating between $0.5/n$ and $1.5/n$; and the \emph{linear staircase}
	\[
		p_k=\frac1n+\Bigl(k-\frac{n+1}2\Bigr)\varepsilon,\qquad \varepsilon=0.006,
	\]
	rising from $0.73/n$ to $1.27/n$ (both sum to $1$). The comb has a large per-edge gradient
	$\mu_j=\pm1/n$ of constant magnitude: every $N\equiv0\pmod{10}$ sits at the lattice phase
	$\{N\mu_j\}=0$, and off that residue class the phase moves and the leading error can change
	sign (rows $N=33,105,303,1003$, phases $0.3$ and $0.5$). The staircase has a small constant gradient $\mu_j=\varepsilon$, whose
	phase drifts slowly with $N$ and makes the oscillation visible as the non-monotone errors in
	its rows.
	\begin{table}[ht]
	\centering\small
	\begin{tabular}{l|r|r|r|r|r}
		distribution & $N$ & exact $\E(d_{\mathbf p})$ & err.\ of \eqref{eq:E-TVC-leading} & with $c_j/N$ & with $c_j^{(2)}/N^2$ \\\hline
		comb       & 30   & $16.853$  & $+1.73\%$  & $+0.027\%$  & $+0.00003\%$ \\
		& 100  & $31.144$  & $+0.51\%$  & $+0.002\%$  & $+0.000008\%$ \\
		& 300  & $54.124$  & $+0.17\%$  & $+0.0003\%$ & $<10^{-6}\%$ \\
		& 1000 & $98.933$  & $+0.050\%$ & $<10^{-4}\%$ & $<10^{-7}\%$ \\
		& 33   & $18.009$  & $-0.15\%$  & $+0.0033\%$ & $+0.00024\%$ \\
		& 105  & $32.122$  & $-0.15\%$  & $-0.0008\%$ & $-0.000025\%$ \\
		& 303  & $54.495$  & $-0.017\%$ & $+0.00004\%$ & $+3\cdot10^{-7}\%$ \\
		& 1003 & $99.136$  & $-0.0051\%$ & $+3\cdot10^{-6}\%$ & $+8\cdot10^{-9}\%$ \\\hline
		staircase  & 30   & $17.452$  & $+0.47\%$  & $+0.025\%$  & $+0.0025\%$ \\
		& 100  & $32.046$  & $-0.10\%$  & $+0.0001\%$ & $+0.000005\%$ \\
		& 300  & $55.430$  & $+0.034\%$ & $+0.0002\%$ & $+0.000002\%$ \\
		& 1000 & $101.185$ & $+0.051\%$ & $<10^{-4}\%$ & $<10^{-6}\%$ \\
	\end{tabular}
	\caption{Successive truncations of \eqref{eq:E-TVC} --- leading term
	\eqref{eq:E-TVC-leading}, the $O(N^{-1})$ correction, and the $O(N^{-2})$ correction
	--- against the exact $\E(d_{\mathbf p})$, computed from the exact mixture; comb (at the lattice
	phase $\{N\mu_j\}=0$, $N\equiv0\pmod{10}$, and off it) and linear staircase at $n=10$.}
	\label{tab:E-TVC-leading}
	\end{table}

% =====================================================================
\section{Detailed proof of the per-edge expansion}
	\label{sec:E-D1-detailed}

	This section proves Lemma~\ref{lem:E-D1-asymp-general} in full. The argument is
	laid out as a sequence of explicitly stated steps rather than compressed into a single
	proof environment, because it combines several distinct mechanisms: a weighted local
	Edgeworth estimate, a corner Euler--Maclaurin summation, and the distributional analysis of a kink.
	The weighted local estimate is standard in method, but the particular form needed
	here is stated and proved below for the present three-point increment law.

	\paragraph{Notation for this section} Fix the edge $j$ and suppress the index
	throughout. Write
	\begin{gather*}
		a := p_j,\quad b := p_{j+1},\quad w := 1 - a - b,\quad
		\mu := \mu_j = b - a,\quad \tau := a + b,\\
		s^2 := s_j^2 = \tau - \mu^2,\qquad \sigma^2 := \sigma_j^2 = N s^2,\\
		D_N := D_j = X_{j+1} - X_j,\qquad \widetilde D_N := \widetilde D_j = D_N - N\mu.
	\end{gather*}
	The hypotheses of the lemma are $a,b\ge0$, $a+b>0$ and $a+b<1$, i.e.\ $w>0$; if one of $a,b$
	vanishes the increment lives on $\{0,1\}$ or $\{-1,0\}$, and everything below holds verbatim (the
	span is still $1$, since $w>0$ and the other weight is positive).
	Each of the $N$ independent multinomial trials contributes an increment
	$\xi \in \{-1, 0, +1\}$ with
	\[
		\Pr(\xi = -1) = a,\qquad \Pr(\xi = +1) = b,\qquad \Pr(\xi = 0) = w,
	\]
	and $D_N = \sum_{t=1}^N \xi_t$. The quantity to expand is exactly a weighted lattice
	sum,
	\begin{equation}\label{eq:exp-as-sum}
		\E|\widetilde D_N| = \sum_{m\in\mathbb Z} |m - N\mu|\,\Pr(D_N = m),
	\end{equation}
	and the whole proof is the asymptotic evaluation of this sum.

\subsection{Moments and cumulants of one increment}

	Since $(\pm1)^{2\ell} = 1$, $(\pm1)^{2\ell+1} = \pm1$, and $0^k = 0$ for $k\ge1$, the raw moments
	are $\E(\xi^{2\ell})=a+b=\tau$ for every $\ell\ge1$ and $\E(\xi^{2\ell+1})=b-a=\mu$ for every
	$\ell\ge0$; in particular
	\[
		m_1=\mu,\quad m_2=\tau,\quad m_3=\mu,\quad m_4=\tau,\quad
		m_5=\mu,\quad m_6=\tau.
	\]
	Hence $\Var(\xi) = \tau - \mu^2 = s^2$, and by independence
	$\Var(D_N) = N s^2 = \sigma^2$, strictly positive because $\xi$ is
	non-degenerate ($w>0$ and $a+b>0$). The third cumulant equals the third central moment,
	\[
		\kappa_3 = \E(\xi-\mu)^3 = m_3 - 3\mu m_2 + 2\mu^3 = \mu(1 - 3\tau + 2\mu^2),
	\]
	and the fourth central moment is
	\[
		\E(\xi-\mu)^4
		=m_4-4\mu m_3+6\mu^2m_2-4\mu^3m_1+\mu^4
		=\tau-4\mu^2+6\tau\mu^2-3\mu^4.
	\]
	The fourth cumulant is $\kappa_4 = \E(\xi-\mu)^4 - 3 s^4$, so
	\[
		\kappa_4 = (\tau - 4\mu^2 + 6\tau\mu^2 - 3\mu^4) - 3(\tau-\mu^2)^2
		= \tau - 3\tau^2 - 4\mu^2 + 12\tau\mu^2 - 6\mu^4,
	\]
	as quoted in Lemma~\ref{lem:E-D1-asymp-general}.

	For the second correction in the Edgeworth expansion we also need the fifth and sixth
	cumulants. A convenient way to compute them is to expand the cumulant generating
	function
	\[
		K(t)=\log \E e^{t\xi}
		=\log\bigl(ae^{-t}+w+be^t\bigr)
		=\sum_{r\ge1}\kappa_r\,\frac{t^r}{r!}.
	\]
	Inserting the raw-moment expansion
	\[
		\E e^{t\xi}
		=1+\mu t+\frac{\tau t^2}{2}+\frac{\mu t^3}{6}
		+\frac{\tau t^4}{24}+\frac{\mu t^5}{120}
		+\frac{\tau t^6}{720}+O(t^7)
	\]
	and expanding $\log(1+u)$ through $t^6$ gives
	\begin{align*}
		\kappa_5
		&=\mu\bigl(24\mu^4-60\mu^2\tau+20\mu^2+30\tau^2-15\tau+1\bigr),\\
		\kappa_6
		&=-120\mu^6+360\mu^4\tau-120\mu^4-270\mu^2\tau^2
		+150\mu^2\tau-16\mu^2+30\tau^3-15\tau^2+\tau.
	\end{align*}
	Cumulants add over sums of independent variables, so
	$\kappa_\ell(D_N) = N\kappa_\ell(\xi)$; standardizing by $\sigma = s\sqrt N$,
	\begin{align*}
		\frac{\kappa_3(D_N)}{\sigma^3} &= \frac{\lambda_3}{\sqrt N},&
		\frac{\kappa_4(D_N)}{\sigma^4} &= \frac{\lambda_4}{N},\\
		\frac{\kappa_5(D_N)}{\sigma^5} &= \frac{\lambda_5}{N^{3/2}},&
		\frac{\kappa_6(D_N)}{\sigma^6} &= \frac{\lambda_6}{N^2},
		\qquad \lambda_r := \frac{\kappa_r}{s^r}.
	\end{align*}
	Thus the direct cumulant contributions for $r=3,4,5,6$ have orders
	$N^{-1/2}$, $N^{-1}$, $N^{-3/2}$, and $N^{-2}$. The products of lower cumulants
	that appear in $P_2,P_3,P_4$ below come from exponentiating the logarithmic
	characteristic-function expansion.

\subsection{Span, aperiodicity, and the Edgeworth approximant}

	The increment is supported on $\{-1, 0, 1\}$, whose differences generate $\mathbb Z$,
	so its maximal lattice span is $1$. Equivalently, the per-trial characteristic
	function $\psi(\theta)=ae^{-i\theta}+w+be^{i\theta}$ is a convex combination, with positive
	weights $w$ and at least one of $a,b$, of the unit vectors $e^{-i\theta},1,e^{i\theta}$, so
	$|\psi(\theta)|=1$ forces those with positive weight to coincide, i.e.\ $\theta\in2\pi\mathbb Z$. This is
	aperiodicity, and it is exactly where $w > 0$ enters: if $w = 0$ the walk uses only
	$\pm1$ and $D_N$ carries the parity of $N$, a span-$2$ phenomenon.

	Let $\phi(z) = (2\pi)^{-1/2} e^{-z^2/2}$ and let $H_k$ denote the probabilists' Hermite
	polynomials, fixed by $\phi^{(k)}=(-1)^k H_k\phi$; the low ones used below are
	$H_3=z^3-3z$, $H_4=z^4-6z^2+3$, $H_6=z^6-15z^4+45z^2-15$, and $H_5,\dots,H_{15}$ are the
	standard higher members. The one-dimensional Edgeworth inversion of the cumulant expansion of
	$S_N = (D_N - N\mu)/\sigma$ gives, through relative order $N^{-5/2}$,
	\begin{equation}\label{eq:edgeworth-pseudo}
		f_N(x) = \frac1\sigma\,\phi(z)\Bigl[1+\sum_{r=1}^{5}N^{-r/2}P_r(z)\Bigr],
		\qquad z = \frac{x - N\mu}{\sigma},
	\end{equation}
	where each $P_r$ is a polynomial in $z$ of parity $(-1)^r$ whose coefficients are polynomials in
	the standardized cumulants $\lambda_3,\lambda_4,\dots$ ($\lambda_m:=\kappa_m/s^m$; $\kappa_5,\kappa_6$ from \eqref{eq:kappa5}):
	\begin{gather*}
		P_1=\tfrac{\lambda_3}6 H_3,\qquad
		P_2=\tfrac{\lambda_4}{24}H_4+\tfrac{\lambda_3^2}{72}H_6,\qquad
		P_3=\tfrac{\lambda_5}{120}H_5+\tfrac{\lambda_3\lambda_4}{144}H_7+\tfrac{\lambda_3^3}{1296}H_9,\\
		P_4=\tfrac{\lambda_6}{720}H_6+\Bigl(\tfrac{\lambda_3\lambda_5}{720}+\tfrac{\lambda_4^2}{1152}\Bigr)H_8
		+\tfrac{\lambda_3^2\lambda_4}{1728}H_{10}+\tfrac{\lambda_3^4}{31104}H_{12},
	\end{gather*}
	and $P_5$ is the corresponding \emph{odd} polynomial in the cumulants
	$\lambda_3,\dots,\lambda_7$, a combination of the Hermite polynomials of odd degree
	$7$ through $15$. Its explicit coefficients are never needed:
	$P_5$ is carried only so that the local estimate below reaches relative order $N^{-3}$; its
	smooth $N^{-5/2}$ mode vanishes by oddness and its corner terms start at $N^{-7/2}$, so it does
	not reach the retained order (Step~5).

	$P_1,P_2$ come from the third and fourth cumulants (the $H_6$ term from the square of the
	third-cumulant contribution when the logarithmic expansion is exponentiated), $P_3,P_4$ from
	the fifth and sixth, $P_5$ from the seventh; these are the standard Edgeworth polynomials in the
	cumulant-to-Hermite normalization of Petrov~\cite[Chs.~VI--VII]{petrov1975} (equivalently
	Bhattacharya--Rao~\cite{bhattacharya_rao}), obtained by grouping the terms of
	$\exp\bigl(\sum_{r\ge3}\kappa_r(it)^r/r!\bigr)$ by total order in $N^{-1/2}$. Under Fourier
	inversion the factors $(it)^k e^{-t^2/2}$ become $H_k(z)\phi(z)$. Every $P_r$ is a
	combination of Hermite polynomials $H_m$ with $m\equiv r\pmod 2$, which is the source of the
	parity $(-1)^r$ used repeatedly below. The function $f_N$ is
	not a density --- $D_N$ is a lattice variable --- but a smooth function whose values at integers
	approximate the lattice probabilities $\Pr(D_N=m)$.

\begin{lemma}[Weighted local Edgeworth input]\label{lem:weighted-local-edgeworth}
	For the fixed three-point increment law above, with $a,b\ge0$, $a+b>0$, $w>0$ (hence maximal
	span $1$), the Edgeworth approximant \eqref{eq:edgeworth-pseudo} satisfies
	\begin{equation}\label{eq:weighted-local-detailed}
		\sup_{m\in\mathbb Z}\,(1 + |z_m|^6)\,\sigma\,
		\bigl|\Pr(D_N = m) - f_N(m)\bigr| = O(N^{-3}),
		\quad z_m := \frac{m - N\mu}{\sigma}.
	\end{equation}
	The implied constant may depend on the fixed edge probabilities $(a,b,w)$ but is
	uniform in the lattice point $m$. If $(a,b,w)$ ranges over a compact subset of
	$\{a,b\ge0,\ a+b>0,\ w>0,\ a+b+w=1\}$, the same estimate is uniform over that compact set.
\end{lemma}

	For a fixed law the estimate is a special case of the lattice local expansion with polynomial
	weight (Bhattacharya--Rao~\cite[Ch.~5, \S22]{bhattacharya_rao}; unweighted:
	Petrov~\cite[Ch.~VII]{petrov1975}); we give the proof because the uniformity over compact sets
	of $(a,b,w)$, needed in Theorem~\ref{thm:E-TVC}, is not stated there.

\begin{proof}
	We give the standard Fourier proof in the form needed here; for the general theory of
	weighted local Edgeworth expansions see Bhattacharya--Rao~\cite{bhattacharya_rao}. The
	point worth recording is the polynomially weighted form: the later absolute-value
	sum uses the estimate after multiplication by \(1+|z|^6\), so the Fourier
	remainder has to be controlled together with six derivatives. The derivative
	estimates this requires are the content of Region~A below, and are stated there in
	full. The
	centered per-trial characteristic function is
	$\widehat\psi(\theta)=\E e^{i\theta(\xi-\mu)}=e^{-i\mu\theta}\psi(\theta)$, with
	$\psi$ as above and $\widehat\psi(0)=1$. The centered factor need not be
	$2\pi$-periodic unless $\mu\in\mathbb Z$, but $|\widehat\psi|=|\psi|$, so the
	span-$1$ aperiodicity established above gives
	$|\widehat\psi(\theta)|<1$ for $0<|\theta|\le\pi$. Since the law of the increment is
	fixed, all constants below may depend on $(a,b,w)$, but not on $N$ or on the lattice
	point $m$.

	\emph{Step 1 (inversion).} Fourier inversion on $\mathbb Z$ and the substitution
	$\theta=t/\sigma$ give
	\begin{equation}\label{eq:wle-inv}
		\sigma\,\Pr(D_N=m)=\frac1{2\pi}\int_{-\pi\sigma}^{\pi\sigma}\widehat\psi(t/\sigma)^N e^{-iz_m t}\,dt,
		\qquad
		\sigma f_N(m)=\frac1{2\pi}\int_{\mathbb R}\chi(t)\,e^{-iz_m t}\,dt,
	\end{equation}
	where, using $\tfrac1{2\pi}\int_{\mathbb R}(it)^k e^{-t^2/2}e^{-izt}\,dt=H_k(z)\phi(z)$, the
	function $\chi$ is the Fourier-side form of the approximant \eqref{eq:edgeworth-pseudo}: through
	relative order $N^{-2}$,
	\begin{equation}\label{eq:wle-chi}
		\chi(t)=e^{-t^2/2}\Bigl[1+\tfrac{\lambda_3}{6\sqrt N}(it)^3
		+\tfrac1N\bigl(\tfrac{\lambda_4}{24}(it)^4+\tfrac{\lambda_3^2}{72}(it)^6\bigr)+\cdots\Bigr],
	\end{equation}
	the $N^{-3/2}$, $N^{-2}$ and $N^{-5/2}$ terms carrying $P_3,P_4,P_5$ under
	$H_k(z)\mapsto(it)^k$. Write
	$U(t)=\widehat\psi(t/\sigma)^N\mathbf 1_{\{|t|\le\pi\sigma\}}$ and $R=U-\chi$, so
	\[
		\sigma(\Pr(D_N=m)-f_N(m))=\frac1{2\pi}\int_{\mathbb R}R(t)e^{-iz_m t}\,dt.
	\]

	\emph{Step 2 (far lattice points).} Because $D_N$ is supported on $[-N,N]$ and
	$|N\mu|\le N$, every $m$ with $|m|>2N$ has $\Pr(D_N=m)=0$ and
	$|z_m|\ge N/\sigma=\sqrt N/s$. There $\sigma|\Pr(D_N=m)-f_N(m)|=\sigma|f_N(m)|
	=\phi(z_m)\,|1+\sum_r N^{-r/2}P_r(z_m)|$, and since $(1+|z|^6)\phi(z)$ times any fixed
	polynomial is bounded by $Ce^{-z^2/4}$,
	\[
		\sup_{|m|>2N}(1+|z_m|^6)\,\sigma\,\bigl|\Pr(D_N=m)-f_N(m)\bigr|
		=O\bigl(e^{-N/(4s^2)}\bigr),
	\]
	which is $o(N^{-K})$ for every fixed $K$. So it suffices to prove
	\eqref{eq:weighted-local-detailed} for $|m|\le 2N$, where
	\begin{equation}\label{eq:z-range}
		|z_m|\le\frac{3N}{\sigma}=\frac{3\sqrt N}{s}.
	\end{equation}
	This restriction is what makes the endpoint terms of Step~3 harmless.

	\emph{Step 3 (the weight is six integrations by parts).} The purpose of the
	weight in \eqref{eq:weighted-local-detailed} is to keep the local estimate useful after
	multiplication by functions with polynomial growth. On the Fourier side this weight is
	obtained by differentiating the transform remainder. Since $(-i)^6=-1$, one has
	$z_m^6 e^{-iz_m t}=-\frac{d^6}{dt^6}e^{-iz_m t}$, and six integrations by parts give
	\[
		(1+z_m^6)\,\sigma\bigl(\Pr(D_N=m)-f_N(m)\bigr)
		=\frac1{2\pi}\int_{\mathbb R}\bigl(R(t)-R^{(6)}(t)\bigr)e^{-iz_m t}\,dt+B_N(m),
	\]
	where $B_N(m)$ collects the endpoint terms at $t=\pm\pi\sigma$. The $\chi$ part has no
	boundary contribution because it is integrated over $\mathbb R$ and is rapidly
	decreasing. The $U$ part is integrated only over $[-\pi\sigma,\pi\sigma]$; after six
	integrations by parts its endpoint terms are products of
	\[
		\frac{d^j}{dt^j}\widehat\psi(t/\sigma)^N\Big|_{t=\pm\pi\sigma}
		\qquad\text{and}\qquad
		\frac{d^{5-j}}{dt^{5-j}}e^{-iz_mt}\Big|_{t=\pm\pi\sigma}
		=(-iz_m)^{5-j}e^{\mp i\pi\sigma z_m},
	\]
	$0\le j\le5$, so they carry powers of $z_m$ up to $|z_m|^5$. Each differentiation of
	$\widehat\psi(t/\sigma)^N$ produces at most one factor of order $N/\sigma$ and lowers the
	power of the characteristic function by at most one, whence
	\[
		|B_N(m)|\le CN^5\sigma^{-5}\bigl(1+|z_m|^5\bigr)\,|\widehat\psi(\pi)|^{N-5}.
	\]
	Since $|\widehat\psi(\pm\pi)|=|\psi(\pm\pi)|=|1-2(a+b)|<1$ and, by \eqref{eq:z-range},
	$1+|z_m|^5=O(N^{5/2})$ on the relevant range $|m|\le2N$, the endpoint terms are
	exponentially small: $\sup_{|m|\le2N}|B_N(m)|=o(N^{-K})$ for every fixed $K$. (This is
	the only place where the polynomial weight interacts with the cutoff, and it is why
	Step~2 was needed: without the restriction $|m|\le2N$ the factor $|z_m|^5$ would be
	unbounded.) Taking absolute values,
	\begin{equation}\label{eq:wle-reduce}
		\sup_{|m|\le 2N}(1+|z_m|^6)\,\sigma\,\bigl|\Pr(D_N=m)-f_N(m)\bigr|
		\le\frac1{2\pi}\bigl(\|R\|_{L^1}+\|R^{(6)}\|_{L^1}\bigr)+o(N^{-K}).
	\end{equation}
	Thus the weighted pointwise estimate has been reduced to an
	$L^1$ estimate on the transform remainder. It remains to prove
	$\|R\|_{L^1}+\|R^{(6)}\|_{L^1}=O(N^{-3})$, and we split the real line into four
	regions.

	\emph{Region A: $|t|\le N^{1/8}$.} For $|\theta|$ small the cumulant expansion gives
	\[
		\log\widehat\psi(\theta)
		=-\frac{s^2\theta^2}{2}
		+\sum_{r=3}^{7}\frac{\kappa_r}{r!}(i\theta)^r+r_8(\theta),
		\qquad r_8^{(q)}(\theta)=O(|\theta|^{8-q}),\quad 0\le q\le6,
	\]
	because the three-point characteristic function is analytic near the origin. After the
	scaling $\theta=t/\sigma$, the first four non-Gaussian terms become
	\[
		N\frac{\kappa_3}{3!}\Bigl(\frac{it}{\sigma}\Bigr)^3
		=\frac{\lambda_3}{6\sqrt N}(it)^3,\quad
		N\frac{\kappa_4}{4!}\Bigl(\frac{it}{\sigma}\Bigr)^4
		=\frac{\lambda_4}{24N}(it)^4,
	\]
	and similarly the fifth, sixth and seventh cumulants enter at orders $N^{-3/2}$, $N^{-2}$
	and $N^{-5/2}$.
	The leading non-Gaussian exponent is now uniformly small:
	$N^{-1/2}|t|^3\le N^{-1/8}$ on this region. The logarithmic remainder is
	$Nr_8(t/\sigma)=O(N^{-3}|t|^8)$. Therefore the exponential of the non-Gaussian
	part may be expanded uniformly through relative order $N^{-5/2}$, giving exactly the
	Fourier polynomial $\chi(t)$: the products such as $\lambda_3^2(it)^6/72$,
	$\lambda_3\lambda_4(it)^7/144$, and the higher products in $P_4,P_5$ are produced by
	the ordinary exponential series. All omitted terms, including the logarithmic
	remainder, are bounded by $N^{-3}$ times a fixed polynomial in $t$.

	The same estimate is stable under six differentiations. To see precisely where
	the order is used, write
	\[
	\begin{aligned}
		N\log\widehat\psi(t/\sigma)
		=-\frac{t^2}{2}
		&+N^{-1/2}A(t)+N^{-1}B(t)+N^{-3/2}C(t)\\
		&+N^{-2}D(t)+N^{-5/2}E(t)+\rho_N(t),
	\end{aligned}
	\]
	where \(A,B,C,D,E\) are fixed polynomials of degrees \(3,4,5,6,7\), and
	\[
		\rho_N^{(q)}(t)=O(N^{-3}|t|^{8-q}),\qquad 0\le q\le6 .
	\]
	The Fourier polynomial \(\chi\) is obtained by expanding
	\[
		\exp\{N^{-1/2}A+N^{-1}B+N^{-3/2}C+N^{-2}D+N^{-5/2}E\}
	\]
	and retaining exactly the terms of total order at most \(N^{-5/2}\).  Hence every
	omitted product is $O(N^{-3})$, multiplied by a fixed polynomial in \(t\), and the
	same is true after six differentiations because derivatives of \(A,\dots,E\) only
	reduce degrees and the differentiated \(\rho_N\) satisfies the bound above.
	Equivalently, with \(h(t)=N\log\widehat\psi(t/\sigma)\), the Gaussian part gives
	$h'(t)=O(1+|t|)$, $h''(t)=O(1)$, while $h^{(q)}(t)=O(\sigma^{2-q})$ for
	$3\le q\le6$ on this region. Fa\`a di Bruno's formula therefore yields,
	together with the Gaussian envelope
	$\operatorname{Re}\log\widehat\psi(\theta)\le-\tfrac{s^2}{4}\theta^2$ for small
	$|\theta|$,
	\[
		|R(t)|+|R^{(6)}(t)|\le C\,N^{-3}(1+|t|^{M})e^{-t^2/4}\qquad(|t|\le N^{1/8})
	\]
	for a fixed $M$. The right-hand side is integrable uniformly in $N$, so Region~A
	contributes $O(N^{-3})$ to $\|R\|_{L^1}+\|R^{(6)}\|_{L^1}$.

	\emph{Region B: $N^{1/8}<|t|\le\delta\sigma$.} Choose $\delta>0$ so small that
	$\operatorname{Re}\log\widehat\psi(\theta)\le-\tfrac{s^2}{4}\theta^2$ for
	$|\theta|\le\delta$. Such a $\delta$ exists because the second derivative of
	$\operatorname{Re}\log\widehat\psi(\theta)$ at $0$ is $-s^2<0$. In this annulus,
	\[
		|\widehat\psi(t/\sigma)^N|\le e^{-t^2/4}\le e^{-N^{1/4}/4}.
	\]
	Six derivatives bring only a polynomial loss: the bounds on the derivatives of
	$h(t)=N\log\widehat\psi(t/\sigma)$ and Fa\`a di Bruno's formula give
	$\frac{d^6}{dt^6}e^{h(t)}=e^{h(t)}O((1+|t|)^6)$, up to a fixed power of $\sigma$ on
	the whole annulus. The approximant has the same Gaussian protection:
	$|\chi(t)|+|\chi^{(6)}(t)|\le C(1+|t|^{M_1})e^{-t^2/2}$. Hence Region~B contributes
	\[
		O\bigl(\sigma^{M_2}e^{-N^{1/4}/4}\bigr)=o(N^{-K})
	\]
	for every fixed $K$.

	\emph{Region C: $\delta\sigma<|t|\le\pi\sigma$.} The $\chi$ part and its sixth
	derivative are exponentially small here by the Gaussian tail. For the lattice part set
	\[
		\varrho:=\max_{\delta\le|\theta|\le\pi}|\widehat\psi(\theta)|<1.
	\]
	The strict inequality follows from aperiodicity and compactness of the interval
	$\{\delta\le|\theta|\le\pi\}$. Thus
	$|\widehat\psi(t/\sigma)^N|\le \varrho^N$. By the chain rule, each derivative of
	$\widehat\psi(t/\sigma)^N$ costs at most a power of $N/\sigma$ and lowers the exponent
	of $\varrho$ only by a bounded amount, so
	\[
		|U^{(6)}(t)|\le C N^6\sigma^{-6}\varrho^{N-6}.
	\]
	Integrating over a set of length at most $2\pi\sigma$ gives
	$O(\sigma N^6\varrho^{N-6})=o(N^{-K})$, and the same is easier for $U$ without
	differentiation.

	\emph{Region D: $|t|>\pi\sigma$.} Only the Edgeworth approximant is present, because
	$U$ is supported on $[-\pi\sigma,\pi\sigma]$. Since $\chi$ is a Gaussian times a
	polynomial of fixed degree,
	\[
		\int_{|t|>\pi\sigma}\bigl(|\chi(t)|+|\chi^{(6)}(t)|\bigr)\,dt
		\le C(1+\sigma^{M_3})e^{-\pi^2\sigma^2/2}=o(N^{-K}),
	\]
	using $\sigma^2=Ns^2$.

	Summing the four regions gives
	\[
		\|R\|_{L^1}+\|R^{(6)}\|_{L^1}=O(N^{-3})+o(N^{-K})=O(N^{-3}).
	\]
	With \eqref{eq:wle-reduce} and Step~2, this proves \eqref{eq:weighted-local-detailed}
	uniformly over all $m\in\mathbb Z$.

	\emph{Uniformity.} Every constant above --- the Taylor constant in Region~A, the
	threshold $\delta$ and the contraction factor $\varrho<1$ in Regions~B--C, and the boundary
	value $|1-2(a+b)|<1$ --- can be chosen uniformly when $(a,b,w)$ ranges over a compact
	subset of $\{a,b\ge0,\ a+b>0,\ w>0\}$; there $s^2=\tau-\mu^2$ is also bounded above and away
	from $0$ (at $a=0$ it equals $b(1-b)>0$). Hence
	\eqref{eq:weighted-local-detailed} holds uniformly over that compact set.
\end{proof}

\subsection{Step 1: from the exact lattice law to a smooth sum}

	Define $h(x) := |x - N\mu|\,f_N(x)$. By \eqref{eq:weighted-local-detailed} there is a
	constant $C$ with
	\[
		|\Pr(D_N = m) - f_N(m)| \le C\,N^{-3} / \bigl(\sigma(1 + |z_m|^6)\bigr),
	\]
	and
	since $|m - N\mu| = \sigma|z_m|$,
	\[
		|m - N\mu|\,\bigl|\Pr(D_N = m) - f_N(m)\bigr|
		\le \frac{C}{N^{3}}\,\frac{|z_m|}{1 + |z_m|^6}.
	\]
	The points $z_m$ form a grid of spacing $1/\sigma$, and $g(z) = |z|/(1 + |z|^6)$ is
	nonnegative and integrable (behaving like $|z|$ near $0$ and like $|z|^{-5}$ at
	infinity). To see the uniformity in the grid shift, split the line into unit
	intervals. On the finitely many intervals near the origin, $g$ is bounded and
	there are $O(\sigma)$ grid points; in the tails, compare the points in each unit
	interval with the summable majorant $C(1+|z|)^{-5}$. Thus its sums over
	any grid of spacing $1/\sigma$ are $O(\sigma)$ (its Riemann sums are $O(1)$), uniformly in the
	shift. Summing,
	\begin{equation}\label{eq:step1-detailed}
		\Bigl|\,\E|\widetilde D_N| - \sum_{m\in\mathbb Z} h(m)\,\Bigr|
		\le \frac{C}{N^{3}}\sum_{m\in\mathbb Z}\frac{|z_m|}{1 + |z_m|^6}
		= O\Bigl(\frac{\sigma}{N^{3}}\Bigr).
	\end{equation}
	Because the leading term is of order $\sigma \asymp \sqrt N$, this is a
	\emph{relative} $O(N^{-3})$ error.

\subsection{Step 2: Euler--Maclaurin with a corner}

	Set $h(x):=|x-N\mu|\,f_N(x)$ and $x_0:=N\mu$, so that in the notation of
	Lemma~\ref{lem:corner-EM} below $g=h$ and $u=f_N$. The smooth sum $\sum_m h(m)$ differs from the
	integral $\int h$ only through the kink of the absolute value at $x_0$, which usually lies
	off the integer lattice. The corner form of the Euler--Maclaurin formula below isolates that
	difference and, read at successive orders, delivers the \emph{whole} lattice correction from a
	single computation --- both the leading $B_2$ term of Lemma~\ref{lem:E-D1-asymp-general} and
	the $B_2,B_3,B_4$ terms of \eqref{eq:ell-j}. The only local input is the jump
	structure of the kink, recorded first.

\begin{lemma}[Second derivative of a smooth profile times a kink]\label{lem:kink-detailed}
	Let $u$ be a $C^2$ function with $u, u', u''$ rapidly decaying. In the sense of
	distributions,
	\begin{equation}\label{eq:kink-second-derivative}
		\bigl(|y|\,u(y)\bigr)'' = 2 u(0)\,\delta_0 + |y|\,u''(y) + 2\,\operatorname{sgn}(y)\,u'(y).
	\end{equation}
\end{lemma}

\begin{proof}
	For $y > 0$, $(|y|u)' = u + y u'$; for $y < 0$, $(|y|u)' = -u - y u'$. The one-sided
	limits of the first derivative at $0$ are $u(0)$ and $-u(0)$, a jump of $2u(0)$,
	which contributes the atom $2u(0)\delta_0$ to the second derivative; away from $0$,
	ordinary differentiation gives $|y|u'' + 2\operatorname{sgn}(y) u'$.
\end{proof}

\begin{lemma}[Euler--Maclaurin with one corner]\label{lem:corner-EM}
	Let $u$ be a Schwartz function, $x_0\in\mathbb R$, and $g(x)=|x-x_0|\,u(x)$. Write
	$\tilde B_k(x)=B_k(\{x\})$ for the periodic Bernoulli functions and
	$[g^{(j)}]_{x_0}=g^{(j)}(x_0^+)-g^{(j)}(x_0^-)$ for the jump of the $j$-th derivative at
	$x_0$. Then $g$ is continuous, smooth off $x_0$, and rapidly decaying with all its
	derivatives; the jumps are
	\begin{equation}\label{eq:corner-jumps}
		[g^{(k-1)}]_{x_0}=2(k-1)\,u^{(k-2)}(x_0)\qquad(k\ge2),
	\end{equation}
	and for every integer $K\ge2$
	\begin{equation}\label{eq:corner-EM}
	\begin{aligned}
		\sum_{m\in\mathbb Z}g(m)
		&=\int_{\mathbb R}g(x)\,dx
		+\sum_{k=2}^{K}\frac{(-1)^{k-1}}{k!}\,
		\tilde B_k(x_0)\,[g^{(k-1)}]_{x_0}+R_K,\\
		R_K
		&=\frac{(-1)^{K-1}}{K!}
		\int_{\mathbb R}\tilde B_K(x)\,g^{(K)}(x)\,dx,
	\end{aligned}
	\end{equation}
	the last integral over the two smooth pieces $x\lessgtr x_0$.
\end{lemma}

\begin{proof}
	On each side of $x_0$ one has $g^{(j)}(x)=\pm\bigl((x-x_0)u^{(j)}(x)+j\,u^{(j-1)}(x)\bigr)$ by the
	Leibniz rule (since $(x-x_0)''=0$), with $+$ for $x>x_0$; letting $x\to x_0^\pm$ gives
	$[g^{(j)}]_{x_0}=2j\,u^{(j-1)}(x_0)$, which is \eqref{eq:corner-jumps} with $j=k-1$ (the case
	$k=2$ is \eqref{eq:kink-second-derivative}, where $[g']_{x_0}=2u(x_0)$); $g$ itself is
	continuous.

	The level-one periodic Euler--Maclaurin identity, valid because $g$ is continuous with
	$g,g'\in L^1(\mathbb R)$,
	\[
		\sum_{m\in\mathbb Z}g(m)=\int_{\mathbb R}g+\int_{\mathbb R}\tilde B_1(x)\,g'(x)\,dx,
		\qquad\tilde B_1(x)=\{x\}-\tfrac12,
	\]
	begins the recursion. As $g'$ jumps only at $x_0$, split
	$\int_{\mathbb R}=\int_{-\infty}^{x_0}+\int_{x_0}^{\infty}$ and integrate by parts using
	$\tilde B_1=\tfrac12\tilde B_2'$; the periodic $\tilde B_2$ is continuous, the endpoints at
	$\pm\infty$ vanish by decay, and the two interior endpoints at $x_0$ combine to
	$-\tfrac12\tilde B_2(x_0)[g']_{x_0}$, so that
	\[
		\int_{\mathbb R}\tilde B_1 g'
		=-\tfrac12\tilde B_2(x_0)[g']_{x_0}-\tfrac12\int_{\mathbb R}\tilde B_2 g''.
	\]
	Iterating with $\tilde B_{k-1}=\tfrac1k\tilde B_k'$ peels off, at level $k$, the
	corner term
	\[
		\frac{(-1)^{k-1}}{k!}\tilde B_k(x_0)[g^{(k-1)}]_{x_0}
	\]
	and passes to $\tfrac{(-1)^{k-1}}{k!}\int\tilde B_k g^{(k)}$; stopping at $k=K$ leaves
	$R_K$. (The level-one identity is the case $k=1$ of the same formula, with the empty
	corner sum and $R_1=\int\tilde B_1g'$; the displayed $k=2$ step above is
	$R_2=-\tfrac12\int\tilde B_2g''$, which fixes the sign of $R_K$.)
\end{proof}

	We apply \eqref{eq:corner-EM} to $g=h=|x-N\mu|\,f_N$, so $u=f_N$ and $x_0=N\mu$. Because
	$f_N=\sigma^{-1}\phi(z)[\,\cdots]$ with $z=(x-N\mu)/\sigma$, and $\sigma=s\sqrt N$,
	\begin{equation}\label{eq:corner-order}
		\frac{f_N^{(k-2)}(N\mu)}{\sigma}
		=\frac1{\sigma^{k}}\sum_{r\ge0}N^{-r/2}\,(\phi P_r)^{(k-2)}(0)
		\qquad(P_0:=1),
	\end{equation}
	so the corner term of index $k$, read against the Edgeworth profile of index $r$,
	contributes at \emph{relative} order $N^{-(k+r)/2}$ and carries the factor
	$(\phi P_r)^{(k-2)}(0)$. This single bookkeeping rule governs everything that follows.
	Applying Lemma~\ref{lem:corner-EM} with $K=6$, the retained terms $k=2,3,4$ capture every
	contribution through relative order $N^{-2}$, and the omitted corners $k=5,6$ are relative
	$O(N^{-3})$ by \eqref{eq:corner-order} and the parity rule of Step~5 (for $k=5$ the term
	$r=0$ would sit at $N^{-5/2}$, but it carries $\phi'''(0)=0$).

	The same scaling gives the required bound for the integral remainder. On the two smooth
	pieces \(x\lessgtr x_0\), writing \(y=x-x_0\),
	\[
		\frac{d^6}{dx^6}\bigl(|y|f_N(x)\bigr)
		=\operatorname{sgn}(y)\bigl(y f_N^{(6)}(x)+6f_N^{(5)}(x)\bigr).
	\]
	Each derivative of \(f_N=\sigma^{-1}\phi(z)[\cdots]\), \(z=y/\sigma\), contributes one
	factor \(\sigma^{-1}\), and the Edgeworth polynomial factors have Gaussian decay.
	Consequently
	\[
		\int_{\mathbb R}|y|\,|f_N^{(6)}(x)|\,dx
		+\int_{\mathbb R}|f_N^{(5)}(x)|\,dx=O(\sigma^{-5}).
	\]
	Since the periodic Bernoulli function is bounded, \(R_6=O(\sigma^{-5})\), which is
	relative \(O(\sigma^{-6})=O(N^{-3})\) compared with the leading term of order
	\(\sigma\). The integral $\int h$ and the relevant corner terms are evaluated in
	Steps~3 and~4.

\subsection{Step 3: the smooth integral}

	The integral in \eqref{eq:corner-EM} carries the smooth (Edgeworth) part; it does not see
	where the lattice sits relative to the kink. With $z=(x-N\mu)/\sigma$ the factor $\sigma^{-1}$
	in $f_N$ cancels the Jacobian $dx=\sigma\,dz$, so by \eqref{eq:edgeworth-pseudo}
	\[
		\int_{\mathbb R}h\,dx=\int_{\mathbb R}|x-N\mu|\,f_N(x)\,dx
		=\sigma\int_{\mathbb R}|z|\,\phi(z)\Bigl[1+\sum_{r=1}^5 N^{-r/2}P_r(z)\Bigr]dz.
	\]
	Since $|z|\phi$ is even and $P_r$ has parity $(-1)^r$, the odd terms $r=1,3,5$ integrate to
	zero (the odd-order Edgeworth polynomials integrate to zero against the even weight $|z|\phi$), leaving $r=0,2,4$. Writing
	$I_{2\ell}:=\int_{\mathbb R}|z|\phi(z)H_{2\ell}(z)\,dz$ and $A:=\sqrt{2/\pi}$, one has
	$I_0=A$ and $I_{2\ell}=(-1)^{\ell-1}(2\ell-3)!!\,A$ for $\ell\ge1$, i.e.
	\[
		\begin{gathered}
			I_4=-A,\quad I_6=3A,\quad I_8=-15A,\\
			I_{10}=105A,\quad I_{12}=-945A.
		\end{gathered}
	\]
	Hence, with $P_2=\tfrac{\lambda_4}{24}H_4+\tfrac{\lambda_3^2}{72}H_6$ and $P_4$ from
	\eqref{eq:edgeworth-pseudo},
	\[
		\int|z|\phi\,P_2=\frac{\lambda_3^2-\lambda_4}{24}\,A,
	\]
	\[
		\int|z|\phi\,P_4=\frac{24\lambda_6-120\lambda_3\lambda_5-75\lambda_4^2
		+350\lambda_3^2\lambda_4-175\lambda_3^4}{5760}\,A,
	\]
	the second by inserting $I_6,\dots,I_{12}$. Therefore
	\begin{equation}\label{eq:zero-mode-detailed}
	\begin{gathered}
		\int_{\mathbb R}h\,dx
		=\sqrt{\tfrac2\pi}\,\sigma\Bigl[1+\frac{\lambda_3^2-\lambda_4}{24 N}
		+\frac{c^{(2)}_{\mathrm{sm}}}{N^2}\Bigr],\\
		c^{(2)}_{\mathrm{sm}}
		=\frac{24\lambda_6-120\lambda_3\lambda_5-75\lambda_4^2
		+350\lambda_3^2\lambda_4-175\lambda_3^4}{5760},
	\end{gathered}
	\end{equation}
	the smooth parts of $c_j$ and $c_j^{(2)}$.

\subsection{Step 4: the corner terms}

	The corner sum in \eqref{eq:corner-EM} carries the lattice correction. By
	\eqref{eq:corner-jumps} with $u=f_N$, $[h^{(k-1)}]_{N\mu}=2(k-1)\,f_N^{(k-2)}(N\mu)$, so the
	corner terms are read off the derivatives of the Edgeworth density $f_N$ at its center. From
	$\tfrac{d}{dz}(\phi H_k)=-\phi H_{k+1}$ and $H_{2\ell}(0)=(-1)^\ell(2\ell-1)!!$, $H_{2\ell+1}(0)=0$,
	\begin{equation}\label{eq:fN-derivs-corner}
	\begin{aligned}
		f_N(N\mu)
		&=\frac{\phi(0)}\sigma
		\Bigl[1+\frac{\lambda_4}{8N}-\frac{5\lambda_3^2}{24N}+\cdots\Bigr],\\
		f_N'(N\mu)
		&=-\frac{\phi(0)\lambda_3}{2\sigma^2\sqrt N}+\cdots,\\
		f_N''(N\mu)
		&=-\frac{\phi(0)}{\sigma^3}+\cdots.
	\end{aligned}
	\end{equation}
	Insert these into the $k=2,3,4$ corner terms of \eqref{eq:corner-EM},
	\[
		-\tfrac12\tilde B_2(x_0)\,[h']_{x_0}
		+\tfrac16\tilde B_3(x_0)\,[h'']_{x_0}
		-\tfrac1{24}\tilde B_4(x_0)\,[h''']_{x_0},
		\qquad x_0=N\mu,
	\]
	with jumps $[h']_{x_0}=2f_N(x_0)$, $[h'']_{x_0}=4f_N'(x_0)$, $[h''']_{x_0}=6f_N''(x_0)$, and
	divide by the leading $\sqrt{2/\pi}\,\sigma=2\phi(0)\sigma$. At relative order $N^{-1}$ only
	$k=2$ with the Gaussian $f_N(x_0)=\phi(0)/\sigma$ survives:
	\begin{equation}\label{eq:corner-leading}
		\frac{-B_2(\{N\mu\})\,f_N(x_0)}{\sqrt{2/\pi}\,\sigma}
		=-\frac{B_2(\{N\mu\})}{2\,s^2\,N},
	\end{equation}
	the leading lattice term of $c_j$. Carrying the Edgeworth corrections to $f_N(x_0)$ and the
	(Edgeworth-order) $f_N'(x_0),f_N''(x_0)$ of \eqref{eq:fN-derivs-corner} into all three corner
	terms gives, at relative order $N^{-2}$, the lattice part
	\begin{equation}\label{eq:ell-j-detailed}
		\ell_j(N)=\frac{5\lambda_3^2-3\lambda_4}{48\,s^2}\,B_2(\{N\mu\})
		-\frac{\lambda_3}{6\,s^3}\,B_3(\{N\mu\})
		+\frac1{8\,s^4}\,B_4(\{N\mu\}),
	\end{equation}
	the polynomial $B_3$, odd about $t=\tfrac12$, carried by the skewness $\lambda_3$ through
	$f_N'(x_0)$. The truncation is self-certifying: after retaining the displayed $k=2,3,4$ terms,
	the unretained part consists of the $k=5,6$ corner terms and the Edgeworth pieces of index $r$
	with $k+r\ge5$ in the $k=2,3,4$ corner terms, all relative $O(N^{-3})$ by
	\eqref{eq:corner-order} and the parity rule of Step~5, together with the Euler--Maclaurin
	remainder
	\[
		|R_6|\le\frac1{6!}\,\|\tilde B_6\|_\infty\!\int_{\mathbb R}|h^{(6)}(x)|\,dx=O(\sigma^{-5}),
	\]
	a relative $O(N^{-3})$. The harmonics $B_5,B_6,\dots$ therefore \emph{enter} only at relative
	order $N^{-3}$; so exactly $B_2,B_3,B_4$ occur through $N^{-2}$.
	Geometrically the kink sits at $x=N\mu$ while the law is sampled at integers, and the corner
	Euler--Maclaurin encodes the dependence on the fractional position $\{N\mu\}$ inside one
	lattice cell through the periodic Bernoulli polynomials.

\subsection{Step 5: parity, and the assembly of both orders}

	The expansion has only two sources: the smooth zero mode of Step~3 and the corner terms
	of Step~4. Each carries a parity, and together they force the whole expansion onto the
	integer powers of $N^{-1}$.

\begin{lemma}[Parity]\label{lem:parity}
	In the expansion of $\E|\widetilde D_N|$ relative to its leading term
	$\sqrt{2/\pi}\,\sigma$:
	\begin{enumerate}
		\item the smooth zero mode of index $r$ sits at relative order $N^{-r/2}$ and carries
		$\int_{\mathbb R}|z|\phi(z)P_r(z)\,dz$, which vanishes for every \emph{odd} $r$;
		\item the corner term of index $k\ge2$, read against the Edgeworth profile of index
		$r\ge0$, sits at relative order $N^{-(k+r)/2}$ and carries $(\phi P_r)^{(k-2)}(0)$,
		which vanishes whenever $k+r$ is \emph{odd}.
	\end{enumerate}
	Consequently every half-integer relative order vanishes and, through the order retained here,
	the expansion proceeds in integer powers of $N^{-1}$; the argument applies verbatim at every
	order.
\end{lemma}

\begin{proof}
	(i) $|z|\phi$ is even and $P_r$ has parity $(-1)^r$, so the integrand is odd for odd $r$.
	(ii) The order is \eqref{eq:corner-order}. The function $\phi P_r$ has parity $(-1)^r$, so
	its $(k-2)$-nd derivative has parity $(-1)^{r+k-2}=(-1)^{r+k}$; an odd function vanishes at
	$0$, which is the case $k+r$ odd. A half-integer relative order $N^{-\ell/2}$ with $\ell$
	odd can be reached only with $r=\ell$ odd in (i), or with $k+r=\ell$ odd in (ii); both
	vanish.
\end{proof}

	The lemma explains the pattern already visible in the answer. At relative order $N^{-1}$
	only $(k,r)=(2,0)$ and the smooth $r=2$ survive, giving $c_j(N)$. At $N^{-2}$ the surviving
	pairs are $(2,2)$, $(3,1)$, $(4,0)$ --- exactly the three lattice terms $B_2,B_3,B_4$ of
	$\ell_j$ --- together with the smooth $r=4$. At $N^{-1/2}$, $N^{-3/2}$ and $N^{-5/2}$
	nothing survives at all. This is why the second correction is followed by $O(N^{-3})$ and
	not by $O(N^{-5/2})$: there is no $N^{-5/2}$ term to carry.

	Combining Step~1 \eqref{eq:step1-detailed} (the relative $O(N^{-3})$ passage to the smooth
	sum), the corner Euler--Maclaurin identity \eqref{eq:corner-EM} with $K=6$, the smooth integral of
	Step~3 \eqref{eq:zero-mode-detailed}, the corner terms
	\eqref{eq:corner-leading}--\eqref{eq:ell-j-detailed} of Step~4, and Lemma~\ref{lem:parity},
	\[
		\E|\widetilde D_N|
		=\sqrt{\tfrac2\pi}\,\sigma
		\Bigl[1+\frac{c_j(N)}{N}+\frac{c_j^{(2)}(N)}{N^2}+O(N^{-3})\Bigr],
	\]
	with, on restoring the edge index $j$,
	\[
		c_j(N)=\frac{\lambda_3^2-\lambda_4}{24}-\frac{B_2(\{N\mu_j\})}{2s_j^2},
		\qquad
		c_j^{(2)}(N)=c^{(2)}_{\mathrm{sm}}+\ell_j(N),
	\]
	the smooth $c^{(2)}_{\mathrm{sm}}$ from \eqref{eq:zero-mode-detailed} and the lattice
	$\ell_j$ from \eqref{eq:ell-j-detailed}. This is exactly \eqref{eq:E-D1-asymp-general}, Lemma~\ref{lem:E-D1-asymp-general}; the $1/N$ and $1/N^2$ brackets are its two corrections. This completes the detailed proof, from a single corner Euler--Maclaurin computation.

% =====================================================================
\section{The statistic $d_{\mathbf r}$ under a possibly different sampling law}
	\label{sec:E-dXY}

	So far --- and in Part~1 for the uniform law --- the expectation of the discrete total
	variation was computed in the matched case, $d=d_{\mathbf u}$ and $d_{\mathbf p}$, where the
	sampling and reference laws coincide. Now we treat a more general situation, the one met under
	an alternative: the data $\mathbf X$ comes from the sampling distribution $\mathbf p$, while
	the statistic is centered at the reference distribution $\mathbf r$. Then
	the quantity inside an individual absolute value no longer has mean zero,
	so that the symmetry on which the previous analysis rested is lost.

\begin{definition}[Reference-centered statistic $d_{\mathbf r}$]\label{def:dXY}
	Let $\mathbf X\sim\Mult(N;\mathbf p)$ be the observed vector and let
	$\mathbf r=(r_1,\ldots,r_n)$ be the reference probability vector, fixed by the null. The
	statistic is $d_{\mathbf r}$ of \eqref{eq:dr-def},
	\begin{equation}\label{eq:dXY-def}
		d_{\mathbf r}=\sum_{i=1}^{n-1}\bigl|D_i-N(r_{i+1}-r_i)\bigr|,\qquad D_i=X_{i+1}-X_i,
	\end{equation}
	now viewed under the sampling law $\mathbf p$. Throughout, $\mathbf r$ is a fixed, fully
	specified (\emph{simple}) null; parameter estimation or a composite $\mathbf r$ is not treated
	here, since centering at an estimated law would change the analysis.
\end{definition}

	\paragraph{Special cases}
	\begin{itemize}
		\item If $\mathbf p=\mathbf r=\mathbf u:=(1/n,\dots,1/n)$ is uniform, then $d_{\mathbf r}$ is the
		linear discrete total variation under the uniform null; its expectation was computed in
		\cite[Theorem~3.2]{banic_elezovic_p1}.
		\item If $\mathbf p=\mathbf r$ is arbitrary, one obtains the matched statistic
		$d_{\mathbf p}$ from \S\ref{ssec:tvc-setup}.
		\item In the general case the sampling law $\mathbf p$ and the reference law
		$\mathbf r$ may differ. The statistic centers the observed differences
		$X_{i+1}-X_i$ around the reference shift $N(r_{i+1}-r_i)$; the case
		$\mathbf r=\mathbf u$ under a non-uniform sampling law is included here.
	\end{itemize}

	\paragraph{Statistical interpretation} In tests of the hypothesis
	``$\mathbf X \sim \Mult(N; \mathbf r)$'' the statistic has a distribution that
	depends both on $\mathbf p$ (the sampling distribution) and on $\mathbf r$ (the reference law):
	the law of $d_{\mathbf r}$ under $\mathbf p$ is the relevant one under the alternative (since
	$\mathbf X$ has sampling probability $\mathbf p$, not $\mathbf r$). Under $H_0$: $\mathbf p = \mathbf r$,
	$d_{\mathbf r}$ reduces to the matched statistic $d_{\mathbf p}$. Under $H_1$: $\mathbf p \ne \mathbf r$,
	$d_{\mathbf r}$ accumulates the shifts between the sampling and reference trends.

	\paragraph{The key difference from $\chi^2$} The Pearson $\chi^2$
	statistic~\cite{pearson1900} has, under the usual fixed-$n$ regularity
	conditions (in particular all $r_i>0$), the \emph{asymptotic} null distribution
	$\chi^2_{n-1}$, which does
	\emph{not depend} on the reference vector $\mathbf r$.
	This is the reason why $\chi^2$ is ``general-purpose''. The distribution of $d_{\mathbf r}$
	under $H_0$ depends on $\mathbf r$ through the whole adjacent-edge structure: the
	differences $r_{i+1}-r_i$ place the kinks, while the adjacent sums $r_i+r_{i+1}$ enter the
	edge variances, the null values $s_i^2=(r_i+r_{i+1})-(r_{i+1}-r_i)^2$ of
	Lemma~\ref{lem:cov-Dj-multinomial}.
	It is the \emph{differential} structure that carries the shape information the test is
	sensitive to, but the null law must in any case be computed for each $\mathbf r$
	separately. Consequently --- a heuristic suggested by the experiments
	of~\cite{banic_elezovic_ct}, not a power theorem proved here --- a test that rejects for large
	$d_{\mathbf r}$ can be expected to gain over $\chi^2$ mainly when the edge mismatches
	$\delta_i=(p_{i+1}-p_i)-(r_{i+1}-r_i)$ alternate in sign, i.e.\ for a period-$2$ (comb)
	departure; in~\cite{banic_elezovic_ct} the gain is already lost at periods $4$ and $6$, and for
	smooth mismatches $\chi^2$ is preferable. There is no blind direction: writing
	$\mathbf e=\mathbf p-\mathbf r$, one has
	$\sum_i|\delta_i|\ge\max\mathbf e-\min\mathbf e\ge2\|\mathbf e\|_1/n$; the first inequality is an
	equality for every monotone departure (a linear ramp among them), and both are equalities for a
	balanced single step $\mathbf e=(a,\dots,a,-a,\dots,-a)$ with $n$ even, which is the quantitative
	form of ``smooth mismatches are the weak directions.''

% =====================================================================
\subsection{$d_{\mathbf r}$: asymmetric centering}
	\label{ssec:dXY-setup}

	Let $\mathbf X = (X_1, \ldots, X_n) \sim \Mult(N; \mathbf p)$ be the observed counts,
	and $\mathbf r$ the reference probabilities (Def.~\ref{def:dXY}).
	From (\ref{eq:dXY-def}),
	\begin{equation}\label{eq:dXY-recall}
		d_{\mathbf r} = \sum_{i=1}^{n-1}\bigl|D_i - N(r_{i+1} - r_i)\bigr|,
		\qquad D_i = X_{i+1} - X_i.
	\end{equation}

	\paragraph{Splitting into fluctuation and shift} Let us center $D_i$ around
	the \emph{sampling} mean $\E[D_i] = N(p_{i+1} - p_i)$, introducing
	$\tilde D_i := D_i - N(p_{i+1} - p_i)$ with $\E[\tilde D_i] = 0$. Then
	\begin{equation}\label{eq:dXY-split}
		D_i - N(r_{i+1} - r_i) = \tilde D_i + N\delta_i,
		\qquad
		\delta_i := (p_{i+1} - p_i) - (r_{i+1} - r_i).
	\end{equation}
	We call the quantity $\delta_i$ the \emph{gradient mismatch} at edge $i$ (it is the first
	difference of the bin-level departure $\mathbf p-\mathbf r$; it is not to be confused with the
	Dirac distributions of \S\ref{sec:E-D1-detailed} and of the proof of
	Theorem~\ref{thm:dXY-subleading}). Therefore
	\begin{equation}\label{eq:dXY-folded}
		d_{\mathbf r} = \sum_{i=1}^{n-1}\bigl|\tilde D_i + N\delta_i\bigr|,
		\qquad \E[\tilde D_i] = 0,\ \Var(\tilde D_i) = \sigma_i^2,
	\end{equation}
	where (Lemma~\ref{lem:cov-Dj-multinomial}, under the sampling law $\mathbf p$)
	\begin{equation}\label{eq:sigma-i-dXY}
		\sigma_i^2 = \Var(D_i) = N\bigl[(p_i + p_{i+1}) - (p_i - p_{i+1})^2\bigr].
	\end{equation}

	\paragraph{Where the symmetry is lost} In $d_{\mathbf p}$ (\S\ref{ssec:tvc-setup}) we had
	$\mathbf r = \mathbf p$, so $\delta_i \equiv 0$ and each summand $|\tilde D_i|$
	is the absolute value of a \emph{zero-mean} quantity --- asymptotically
	half-normal, $\E|\tilde D_i| = \sqrt{2/\pi}\,\sigma_i\,(1+o(1))$. As soon as
	$\mathbf r \ne \mathbf p$, a \emph{deterministic shift}
	$N\delta_i$ appears, so $|\tilde D_i + N\delta_i|$ is the absolute value of a variable
	with \emph{nonzero} expectation. The two distributions enter \emph{asymmetrically}:
	$\mathbf p$ controls the fluctuation (the variance $\sigma_i^2$ and the
	conditional binomial parameter), while $\mathbf p$ and $\mathbf r$ together
	determine the mismatch shift.

% =====================================================================
\subsection{$d_{\mathbf r}$: exact computation at finite $N$}
	\label{ssec:dXY-exact}

	The exact mixture-conditioning formula from \S\ref{sssec:mixture-conditioning-general}
	carries over with the same outer variable $K_i=X_i+X_{i+1}$ and the same conditional
	binomial allocation.  The only change is the constant subtracted inside the absolute
	value: it is now the reference shift $N(r_{i+1}-r_i)$ rather than the sampling
	mean shift $N(p_{i+1}-p_i)$.

\begin{proposition}[Exact per-edge expression]\label{prop:dXY-exact}
	For edge $i$ with $p_i+p_{i+1}>0$, under
	$\mathbf X \sim \Mult(N; \mathbf p)$, let
	$K_i := X_i + X_{i+1} \sim \Bin(N, p_i + p_{i+1})$ and
	$q_i := p_{i+1}/(p_i + p_{i+1})$. Then
	\begin{equation}\label{eq:dXY-exact}
		\E\bigl|\tilde D_i + N\delta_i\bigr|
		= \sum_{k=0}^{N}\Pr(K_i = k)\,
		\E\bigl[\,\bigl|2 X_{i+1} - k - N(r_{i+1} - r_i)\bigr| \,\big|\, K_i = k\bigr],
	\end{equation}
	where $X_{i+1} \mid K_i = k \sim \Bin(k, q_i)$. The inner expectation is the
	translated two-bin expectation of Theorem~\ref{thm:F2-translated}, with conditional
	parameter $q_i$ and shift $t_i := N(r_{i+1} - r_i)$.
\end{proposition}

	If $p_i+p_{i+1}=0$, then $X_i=X_{i+1}=0$ almost surely, the edge contribution is
	the deterministic value $N|r_{i+1}-r_i|$, and no mixture calculation is needed.

\begin{proof}
	The marginal of the pair $(X_i, X_{i+1})$ is the three-category law 
	(see Proposition~\ref{prop:k-marginal}):
	$\Mult(N; p_i,p_{i+1},1-p_i-p_{i+1})$.
	Consequently, Proposition~\ref{prop:2bin-factor} gives
	$K_i\sim\Bin(N,p_i+p_{i+1})$ and
	$X_{i+1}\mid K_i=k\sim\Bin(k,q_i)$. On the event $K_i=k$,
	$X_i=k-X_{i+1}$ and therefore
	\[
		D_i-N(r_{i+1}-r_i)
		=2X_{i+1}-k-N(r_{i+1}-r_i).
	\]
	Finally, condition on the disjoint events $\{K_i=k\}$ and apply the law of
	total expectation:
	\[
		\E|D_i-N(r_{i+1}-r_i)|
		=\sum_{k=0}^N\Pr(K_i=k)\,
		\E\bigl[|D_i-N(r_{i+1}-r_i)|\mid K_i=k\bigr].
	\]
	Since \eqref{eq:dXY-split} identifies
	$D_i-N(r_{i+1}-r_i)=\tilde D_i+N\delta_i$, this is exactly
	\eqref{eq:dXY-exact}.
\end{proof}

	\paragraph{Asymmetry in the formula} The conditional binomial parameter
	$q_i = p_{i+1}/(p_i+p_{i+1})$ comes from the sampling law $\mathbf p$ (since
	$X$ is distributed according to $\mathbf p$), and the translation $t_i = N(r_{i+1} - r_i)$ from
	the reference law $\mathbf r$. For $\mathbf r=\mathbf p$ it reduces to
	Proposition~\ref{prop:E-TVC-exact}; for $\mathbf r=\mathbf u$ it gives the uncentered
	$d_{\mathbf u}$ under $\mathbf p$.

% =====================================================================
\subsection{$d_{\mathbf r}$: folded-normal asymptotics}
	\label{ssec:dXY-folded}

	The central asymptotic object is the expectation of the absolute value of a
	normal variable with \emph{nonzero} mean (folded normal).

\begin{definition}[Folded-normal profile]\label{def:G-profile}
	We define
	\begin{equation}\label{eq:G-def}
		G(a) := \sqrt{\frac{2}{\pi}}\,e^{-a^2/2} + a\,\operatorname{erf}\!\Bigl(\frac{a}{\sqrt 2}\Bigr),
		\qquad a \in \mathbb R,
	\end{equation}
	where $\operatorname{erf}(x) = \tfrac{2}{\sqrt\pi}\int_0^x e^{-t^2}dt$.
\end{definition}

\begin{lemma}[Properties of $G$]\label{lem:G-props}
	The function $G$ has the following properties:
	\begin{enumerate}
		\item If $W\sim\mathcal N(m,v^2)$, then $\E|W|=v\,G(m/v)$.
		\item $G$ is \emph{even}: $G(-a) = G(a)$.
		\item $G(0) = \sqrt{2/\pi}$ (half-normal / $d_{\mathbf p}$ case).
		\item $G'(a) = \operatorname{erf}(a/\sqrt2)$, $G''(a) = \sqrt{2/\pi}\,e^{-a^2/2} > 0$
		(strictly convex).
		\item $G(a) = |a| + \sqrt{2/\pi}\,e^{-a^2/2} - |a|\operatorname{erfc}(|a|/\sqrt2)
		= |a| + O(e^{-a^2/2})$ for $|a| \to \infty$; more precisely, by the Mills
		expansion, $G(a)-|a|\sim 2\phi(a)/a^2$ as $|a|\to\infty$ (with $\phi(a)=(2\pi)^{-1/2}
		e^{-a^2/2}$ the standard normal density fixed above), a near-cancellation
		of the two displayed terms.
		\item $G(a) \ge \max\{\sqrt{2/\pi},\, |a|\}$.
	\end{enumerate}
\end{lemma}

\begin{proof}
	(1) Write $W=v(Z+a)$ with $Z\sim\mathcal N(0,1)$ and
	$a=m/v$. Splitting the integral at $z=-a$ gives
	\[
	\begin{aligned}
		\E|Z+a|
		&=\int_{-a}^{\infty}(z+a)\phi(z)\,dz
		-\int_{-\infty}^{-a}(z+a)\phi(z)\,dz\\
		&=2\phi(a)+a\,[1-2\Phi_{\mathcal N}(-a)].
	\end{aligned}
	\]
	Indeed, $\int z\phi(z)\,dz=-\phi(z)$, $\phi(-a)=\phi(a)$, and $\Phi_{\mathcal N}$ denotes the standard
	normal distribution function.
	Now $2\phi(a)=\sqrt{2/\pi}e^{-a^2/2}$ and
	$1-2\Phi_{\mathcal N}(-a)=\operatorname{erf}(a/\sqrt2)$, so
	$\E|W|=vG(a)$. This is the classical folded-normal formula
	(Leone et al.~\cite{leone1961}).
	(2) $e^{-a^2/2}$ even, $a\operatorname{erf}(a/\sqrt2)$ even.
	(3) $G(0) = \sqrt{2/\pi}$.
	(4) By differentiation: $-a\sqrt{2/\pi}e^{-a^2/2} + \operatorname{erf}(a/\sqrt2) + a\sqrt{2/\pi}e^{-a^2/2}
	= \operatorname{erf}(a/\sqrt2)$, so $G'' = \sqrt{2/\pi}e^{-a^2/2} > 0$.
	(5) For $a\ge0$, substitute $\operatorname{erf}(a/\sqrt2)
	=1-\operatorname{erfc}(a/\sqrt2)$ in \eqref{eq:G-def}; evenness gives the
	displayed identity for all $a$. The standard Mills bound
	$\operatorname{erfc}(x)=O(e^{-x^2}/x)$ then gives the stated remainder.
	The sharper asymptotic in the statement follows by inserting the first term of
	the Mills expansion~\cite[\S7.12]{dlmf},
	\[
		\operatorname{erfc}(a/\sqrt2)
		=\frac{\sqrt2}{a\sqrt\pi}e^{-a^2/2}
		\left(1-\frac1{a^2}+O(a^{-4})\right),
		\qquad a\to+\infty.
	\]
	In $G(a)-a=\sqrt{2/\pi}e^{-a^2/2}-a\operatorname{erfc}(a/\sqrt2)$ the leading
	Gaussian terms cancel, leaving
	$2\phi(a)/a^2+O(\phi(a)/a^4)$.  This cancellation explains why, far from the
	matched regime, the folded-normal expectation is extremely close to the
	deterministic absolute shift.
	(6) Evenness, strict convexity and (3) give $G(a)\ge G(0)=\sqrt{2/\pi}$.
	Also, if $Z\sim\mathcal N(0,1)$, then
	$G(a)=\E|Z+a|\ge|\E(Z+a)|=|a|$ by Jensen's inequality. Since both lower bounds
	hold simultaneously, $G(a)\ge\max\{\sqrt{2/\pi},\,|a|\}$.
\end{proof}

	The next theorem is a fixed-pair statement: the sampling law $\mathbf p$ and the
	reference law $\mathbf r$ do not vary with $N$.  Thus each edge is either
	exactly matched, $\delta_i=0$, or has a fixed nonzero gradient mismatch.  The
	transitional scaling $\delta_i=O(N^{-1/2})$ is treated separately below.

\begin{theorem}[Asymptotic $\E(d_{\mathbf r})$]\label{thm:E-dXY}
	Let $n\ge3$ be fixed and let
	$\mathbf X\sim\Mult(N;\mathbf p)$, where $\min_i p_i\ge p_{\min}>0$.
	Let $\mathbf r$ be a fixed reference probability vector.  With
	$\delta_i$ as in \eqref{eq:dXY-split}, $\sigma_i$ as in
	\eqref{eq:sigma-i-dXY}, and $G$ as in \eqref{eq:G-def}, one has
	\begin{equation}\label{eq:E-dXY}
		\E(d_{\mathbf r})
		= \sum_{i=1}^{n-1}\sigma_i\,
		G\!\left(\frac{N\delta_i}{\sigma_i}\right)
		+ O(N^{-1/2}).
	\end{equation}
	More precisely, an edge with $\delta_i=0$ contributes an
	$O(N^{-1/2})$ absolute error, while an edge with fixed
	$\delta_i\ne0$ contributes only $O(e^{-\gamma_iN})$, for some $\gamma_i>0$.
\end{theorem}

\begin{remark}[Scope of the fixed-pair formula]
	Only the sampling law $\mathbf p$ must be interior.  The vector $\mathbf r$ may lie
	on the boundary of the simplex, because the variance $\sigma_i^2$ and the
	conditional binomial parameter in the proof depend only on $\mathbf p$.
	The exponential constant $\gamma_i$ depends on $|\delta_i|$; a statement uniform in $\mathbf r$,
	covering the standardized mismatch $N\delta_i/\sigma_i=\Theta(\sqrt N\,\delta_i)$ in every regime,
	is Theorem~\ref{thm:dXY-subleading}. The case $n=2$ is excluded here because its single
	increment has span $2$; it is handled in \S\ref{ssec:dXY-n2} and in~\cite{elezovic_folded}. The
	fixed-mismatch half holds verbatim at $n=2$ (it uses only Hoeffding), as the explicit formula
	of \S\ref{ssec:dXY-n2} confirms; only the matched half needs the span-one lattice.
\end{remark}

\begin{proof}
	By linearity it is enough to treat the summand $\E|\tilde D_i + N\delta_i|$.
	Write $Z_N:=\tilde D_i/\sigma_i$ and $a_i:=N\delta_i/\sigma_i$ (not the edge probability $a$ of
	\S\ref{sec:E-D1-detailed}), so that
	$|\tilde D_i + N\delta_i| = \sigma_i|Z_N + a_i|$.
	The normal approximation suggests replacing $Z_N$ by a standard normal $Z$,
	which would give exactly $\sigma_i\E|Z+a_i|=\sigma_iG(a_i)$.  The proof below
	only has to justify that this replacement has the stated error in the two fixed
	regimes.  The local regime, where $a_i$ remains bounded but nonzero because
	$\delta_i$ depends on $N$, is deliberately excluded here and handled in
	Theorem~\ref{thm:dXY-subleading}.

	Because $\mathbf p,\mathbf r$ are fixed, each edge is in exactly one of two
	regimes: either $\delta_i=0$ exactly, or $|\delta_i|$ is a fixed positive number.
	The intermediate scaling $\delta_i=O(N^{-1/2})$ is not part of this fixed-pair
	theorem and is treated afterwards.

	\emph{Matched edge, $\delta_i=0$.} Then $a_i=0$, and
	Lemma~\ref{lem:E-D1-asymp-general} gives
	$\E|\tilde D_i|=\sigma_iG(0)+O(N^{-1/2})$: the relative correction is
	$O(N^{-1})$, since the third Edgeworth term vanishes by oddness.
	Here the absolute error is $O(\sigma_i/N)=O(N^{-1/2})$, because the leading
	scale is $\sigma_i=O(\sqrt N)$.  This is the only source of the displayed
	$O(N^{-1/2})$ error when all edges are matched.

	\emph{$|a_i| \to \infty$ (fixed mismatch).} Put $\nu_i := N\delta_i$ and
	$\varepsilon_i:=\operatorname{sgn}(\delta_i)$. Then
	$\varepsilon_i(\tilde D_i+\nu_i)$ has mean $|\nu_i|$, and applying
	$|x|=x+2x^-$ to this signed variable gives the identity
	\[
		\E|\tilde D_i+\nu_i|
		=|\nu_i|+2\E\bigl[(\varepsilon_i(\tilde D_i+\nu_i))^-\bigr].
	\]
	The negative part is nonzero only if
	$\varepsilon_i\tilde D_i<-|\nu_i|=-N|\delta_i|$. Since $\tilde D_i$ is a
	sum of $N$ independent centered, bounded increments and the threshold is a fixed
	positive multiple of $N$, Hoeffding's inequality~\cite{hoeffding1963} gives,
	for example,
	\[
		\Pr\{\varepsilon_i\tilde D_i<-N|\delta_i|\}
		\le \exp\!\left(-\frac{N\delta_i^2}{2}\right).
	\]
	Indeed, each one-trial increment of $D_i$ lies in $\{-1,0,1\}$, and centering
	does not change the length of its range, which is at most $2$. Thus the sum of
	the squared range lengths in Hoeffding's bound is at most $4N$. The negative
	part itself is bounded by a constant times $N$, so its expectation is
	$O(Ne^{-cN})=O(e^{-c'N})$ after decreasing the exponent. Therefore
	$\E|\tilde D_i + N\delta_i| = N|\delta_i| + O(e^{-cN})$, which matches
	$\sigma_i G(a_i) = N|\delta_i| + O(\sigma_i e^{-a_i^2/2})
	= N|\delta_i|+O(e^{-c'N})$ (Lemma~\ref{lem:G-props}(5), with a possibly
	smaller $c'>0$, since $a_i^2$ is proportional to $N$ and the polynomial factor
	$\sigma_i=O(\sqrt N)$ is absorbed into the exponential).
	Thus a fixed mismatched edge is much easier than a matched one: the sign of
	$\tilde D_i+N\delta_i$ is wrong only with exponentially small probability, so the
	absolute value behaves like the deterministic quantity $N|\delta_i|$.
	Summing the finitely many edgewise remainders proves (\ref{eq:E-dXY}).
\end{proof}

	\medskip\noindent\emph{The uncentered statistic.} For the uncentered statistic $d_{\mathbf u}$ (the
	case $\mathbf r=\mathbf u$ of Theorem~\ref{thm:E-dXY}) and a sampling law with all
	$g_i:=p_{i+1}-p_i\neq0$, the theorem gives
	$\E(d_{\mathbf u})=N\sum_i|g_i|+O(e^{-cN})$, and its proof shows more: off an event of
	probability at most $\sum_ie^{-Ng_i^2/2}$, every $|D_i|$ equals $\sgn(g_i)D_i$, so the uncentered
	statistic coincides with the linear functional $\sum_i\sgn(g_i)D_i$ of the counts. This is why
	the counts, and not the statistic, must be centered at a non-uniform null $\mathbf r$ (all
	$r_{i+1}\ne r_i$). Subtracting a constant from $d_{\mathbf u}$ leaves a signed sum; under
	$\mathbf p=\mathbf r\pm\mathbf e$, for $\mathbf e$ small enough that no $g_i$ changes sign, its mean
	moves by $\pm N\sum_i\sgn(r_{i+1}-r_i)(e_{i+1}-e_i)$, up for one sign of the departure and down for
	the other, whereas $d_{\mathbf r}$ has the same leading term $N\sum_i|\delta_i|$ for both signs
	(\S\ref{ssec:dXY-asymmetry}). For $d_{\mathbf u}$ the theorem also shows that the slack $\sum_i\sigma_i$ of
	the Jensen upper bound $\E\|X\|_V\le N\sum_i|p_{i+1}-p_i|+\sum_i\sigma_i$
	of~\cite[Eq.~(36)]{banic_elezovic_tvor} is absent when all $g_i\ne0$:
	$\E\|X\|_V=N\sum_i|p_{i+1}-p_i|+O(e^{-cN})$.

\begin{remark}[Local alternatives]
	If $\mathbf r=\mathbf r_N$ is allowed to depend on $N$ so that
	$a_i=N\delta_i/\sigma_i\to a\ne0$, the off-center kink prevents the third
	Edgeworth term from vanishing. The leading folded-normal formula remains valid,
	but its generic remainder is $O(1)$ absolutely, or $O(N^{-1/2})$ relatively,
	rather than the $O(N^{-1/2})$ absolute remainder in
	Theorem~\ref{thm:E-dXY}.
\end{remark}

	The next theorem records this transitional regime in a form that is uniform in the mismatch.
	It is not needed for the fixed-pair asymptotics above, but it explains what replaces the
	matched expansion when the standardized mismatch stays bounded.

\begin{theorem}[Sub-leading term under local alternatives]\label{thm:dXY-subleading}
	Fix an edge $i$ with $p_i,p_{i+1}>0$ and $p_i+p_{i+1}<1$; in particular this is
	a span-one edge, not the exceptional \(n=2\) case.  Put
	\[
		\mu_i=p_{i+1}-p_i,\qquad
		\tau_i=p_i+p_{i+1},\qquad
		s_i^2=\tau_i-\mu_i^2,
	\]
	and let
	\[
		\lambda_{3,i}=\frac{\kappa_{3,i}}{s_i^3},
		\qquad
		\kappa_{3,i}=\mu_i(1-3\tau_i+2\mu_i^2)
	\]
	be the standardized third cumulant of the one-trial increment. Let $\mathbf r$ be any
	reference probability vector, possibly depending on $N$, and put
	\[
		a_i:=\frac{N\delta_i}{\sigma_i}.
	\]
	Then, with $\phi(z)=(2\pi)^{-1/2}e^{-z^2/2}$ and $G$ as in
	\eqref{eq:G-def},
	\begin{equation}\label{eq:dXY-subleading}
		\E\,|\tilde D_i+N\delta_i|
		= \sigma_i\,G(a_i) - \frac{\lambda_{3,i}}{3}\,s_i\,a_i\,\phi(a_i) + O(N^{-1/2}).
	\end{equation}
	The remainder is uniform in $\mathbf r$ (equivalently, in $a_i\in\mathbb R$); on an edge with fixed
	$\delta_i\ne0$ both $\sigma_iG(a_i)-N|\delta_i|$ and the correction term are exponentially small,
	so \eqref{eq:E-dXY} is recovered there, with the cruder remainder $O(N^{-1/2})$ in place of the
	exponentially small one of Theorem~\ref{thm:E-dXY}.
\end{theorem}

	The correction term in \eqref{eq:dXY-subleading} is odd in $a_i$. It vanishes at
	the matched value $a_i=0$, where the expansion of
	Lemma~\ref{lem:E-D1-asymp-general} resumes, and it decays like
	$a_i\phi(a_i)$ as $|a_i|\to\infty$. The displayed formula records the full
	\(O(1)\) correction in the local regime; the next \(O(N^{-1/2})\) term contains
	the usual Bernoulli-polynomial lattice contribution, now evaluated at the
	off-center kink \(N(r_{i+1}-r_i)\).

	\medskip\noindent\emph{Pitman alternatives.} Since Lemma~\ref{lem:weighted-local-edgeworth} is uniform
	over compact sets of interior edge laws, the theorem also holds when the sampling law
	$\mathbf p=\mathbf p_N$ varies in a compact subset of the interior of the simplex --- in particular for
	Pitman alternatives $\mathbf p_N=\mathbf r+\mathbf h/\sqrt N$ with the null $\mathbf r$ fixed --- provided
	$\sigma_i$, $s_i$ and $\lambda_{3,i}$ are evaluated at $\mathbf p_N$. Expressed through the null values
	$s_i(\mathbf r)$, the leading term acquires the additional $O(1)$ contribution
	$2\phi(a_i)\,\nabla s_i(\mathbf r)\cdot\mathbf h$, which is of the same order as the skewness term
	(on the running edge it is about three times larger for $\mathbf h=(0.5,-1)$; the ratio depends on
	$\mathbf h$). In the two $O(1)$ terms $a_i$ may be evaluated at $\mathbf p_N$ or at $\mathbf r$, since
	the two differ by $O(N^{-1/2})$; in the leading term it must be taken together with $\sigma_i$,
	either $\sigma_i(\mathbf p_N)G(a_i(\mathbf p_N))$ or $\sigma_i(\mathbf r)G(a_i(\mathbf r))$ plus the
	term above.

\begin{proof}
	Standardize $Z_N=\tilde D_i/\sigma_i$ (mean $0$, variance $1$, span-$1$ lattice),
	so $\E|\tilde D_i+N\delta_i|=\sigma_i\,\E|Z_N+a_i|$. The weighted local estimate
	used in \S\ref{sec:E-D1-detailed} applies with the shifted weight $|z+a_i|$ as
	well --- for bounded $a_i$ because $|z+a_i|\le C(1+|z|)$ on such sets of shifts, and for all
	$a_i\in\mathbb R$ by the uniformity paragraph below --- and the grid-shift uniformity of Step~1
	of \S\ref{sec:E-D1-detailed} applies verbatim. Thus the exact lattice expectation may again be replaced by the
	Edgeworth smooth sum up to the stated remainder. The kink of $|z+a_i|$ now sits at
	$z=-a_i$ rather than at the center. The smooth integral of the weighted local
	expansion (\S\ref{sec:E-D1-detailed}, Step~3) gives, with $H_3(z)=z^3-3z$,
	\[
		\int_{\mathbb R}|z+a_i|\,\phi(z)\Bigl[1+\frac{\lambda_{3,i}}{6\sqrt N}H_3(z)+O(N^{-1})\Bigr]dz
		= G(a_i)+\frac{\lambda_{3,i}}{6\sqrt N}\,J_3(a_i)+O(N^{-1}),
	\]
	where the last $O(N^{-1})$ is uniform in $a_i$ (see the uniformity paragraph
	below): the omitted Edgeworth profiles are Gaussian times fixed polynomials, and
	their integrals against $|z+a_i|$ are uniformly bounded. Here
	$J_3(a)=\int_{\mathbb R}|z+a|H_3(z)\phi(z)\,dz$. Since $H_3\phi=-\phi'''$ and,
	in the distributional sense, $(|z+a|)''=2\delta_{-a}$, hence
	$(|z+a|)'''=2\delta'_{-a}$, three integrations by parts give
	\[
	\begin{aligned}
		J_3(a)
		&=-\!\int|z+a|\,\phi'''(z)\,dz
		=\bigl\langle (|\,\cdot\,+a|)''',\phi\bigr\rangle\\
		&=2\langle\delta'_{-a},\phi\rangle
		=-2\phi'(-a)=-2a\,\phi(a).
	\end{aligned}
	\]
	The sign in the second equality follows from three integrations by parts:
	$-\int f\phi'''=\langle f''',\phi\rangle$; the next equality uses
	$\langle\delta'_{x_0},\phi\rangle=-\phi'(x_0)$. Finally
	$\phi'(-a)=a\phi(a)$ because $\phi'(z)=-z\phi(z)$ and $\phi$ is even.
	This is the only new term compared with the matched case.  When $a=0$, the kink
	is back at the center and the integral of the odd Edgeworth polynomial vanishes;
	when $a\ne0$, the kink is off-center and the odd part has a nonzero pairing with
	the absolute-value profile.
	The smooth part therefore equals
	$\sigma_i\bigl[G(a_i)-\tfrac{\lambda_{3,i}}{3\sqrt N}a_i\phi(a_i)\bigr]$, and
	$\sigma_i/\sqrt N=s_i$ produces the stated $O(1)$ term.

	\emph{Uniformity in all of $\mathbb R$.} Since $|\delta_i|\le2$, the Step-1 error with the weight
	$|m-t_i|\le\sigma_i(|z_m|+|a_i|)$ is $O(\sigma_i(1+|a_i|)N^{-3})=O(N^{-2})$; the Edgeworth
	integrals are $\int|z+a|H_k\phi=2H_{k-2}(-a)\phi(a)$ for $k\ge2$ (integration by parts, as for
	$J_3$), bounded in $a$; and the corner terms below are polynomials in $a_i$ times $\phi(a_i)$.

	For completeness, the corner terms are obtained exactly as in
	Step~4 of \S\ref{sec:E-D1-detailed}, but the kink in the unstandardized
	variable now sits at
	$t_i=N(r_{i+1}-r_i)$ because the absolute value is
	$|D_i-t_i|$. The leading smooth profile evaluated at that kink is
	\[
		f_N(t_i)=\frac1{\sigma_i}\phi\!\left(
		\frac{t_i-N(p_{i+1}-p_i)}{\sigma_i}\right)+O(N^{-1})
		=\frac{\phi(a_i)}{\sigma_i}+O(N^{-1}),
	\]
	where the \(O(N^{-1})\) is absolute and uniform in \(a_i\).  We used
	$t_i-N(p_{i+1}-p_i)=-N\delta_i=-\sigma_i a_i$ and the evenness of
	$\phi$. The kink-transform and Bernoulli-series calculation therefore give the
	off-center lattice correction
	\[
		-\frac{\phi(a_i)}{\sigma_i}B_2(\{t_i\})=O(N^{-1/2}).
	\]
	The
	order-$N^{-1}$ Edgeworth terms contribute $\sigma_i\cdot O(N^{-1})=O(N^{-1/2})$;
	both fall into the remainder.
\end{proof}

\begin{corollary}[Mixed regimes under local alternatives]\label{cor:dXY-mixed}
	Fix \(n\ge3\) and an interior sampling law \(\mathbf p\) with \(\min_i p_i>0\).
	Let $\mathbf r=\mathbf r_N$ be a probability vector for each $N$ (matched edges, $a_i=0$, and
	fixed mismatches included). Then
	\[
		\E(d_{\mathbf r})
		=\sum_{i=1}^{n-1}\sigma_i\,G(a_i)
		-\frac13\sum_{i=1}^{n-1}\lambda_{3,i}\,s_i\,a_i\,\phi(a_i)
		+O(N^{-1/2}).
	\]
\end{corollary}

\begin{proof}
	Apply Theorem~\ref{thm:dXY-subleading} edge by edge: on matched edges the
	correction term vanishes ($a_i=0$) and the statement reduces to the matched case
	of Theorem~\ref{thm:E-dXY}; the number of edges is fixed, so the edgewise
	$O(N^{-1/2})$ remainders sum to $O(N^{-1/2})$. This is the configuration met under
	local alternatives: some edges matched, others carrying a bounded standardized mismatch (for
	the Pitman form see the paragraph after Theorem~\ref{thm:dXY-subleading}).
\end{proof}

\begin{remark}[Numerical confirmation]\label{rem:dXY-subleading-num}
	For the running edge $(p_i,p_{i+1})=(0.1,0.3)$ of \S\ref{sec:E-TVC-asymp}
	($\lambda_{3,i}=-1/9$, $s_i=0.6$) under the transitional scaling
	$\delta_i=\alpha s_i/\sqrt N$ (so $a_i\equiv\alpha$), the leading folded-normal $\sigma_iG(\alpha)$
	misses a fixed $O(1)$ amount accounted for by $-\tfrac13\lambda_{3,i}s_i\alpha\phi(\alpha)$
	($=0.00538$ at $\alpha=1$), and the corrected formula \eqref{eq:dXY-subleading} has a residual
	decaying as $N^{-1/2}$; the reference values are computed from the exact mixture of
	Proposition~\ref{prop:dXY-exact} (no Monte Carlo error). For $\alpha=1$, with signed
	residuals (exact minus formula):
	\begin{center}\small
	\begin{tabular}{r|c|c|c}
		$N$ & exact $-\,\sigma_iG(\alpha)$ & residual of \eqref{eq:dXY-subleading} & after adding the lattice term \\\hline
		$100$     & $-1.4\cdot10^{-3}$ & $-6.8\cdot10^{-3}$ & $-3\cdot10^{-5}$ \\
		$1600$    & $+3.7\cdot10^{-3}$ & $-1.7\cdot10^{-3}$ & $-4\cdot10^{-6}$ \\
		$25\,600$ & $+5.0\cdot10^{-3}$ & $-4.2\cdot10^{-4}$ & $-7\cdot10^{-7}$ \\
		$1001$    & $+5.4\cdot10^{-3}$ & $+5.4\cdot10^{-5}$ & $+1.4\cdot10^{-5}$ \\
	\end{tabular}
	\end{center}
	Along $N=100\cdot4^k$ the kink $t_i=0.2N-0.6\sqrt N$ is an integer, so the residual is, to
	within half a percent, the frozen off-center lattice term $-\phi(\alpha)B_2(0)/\sigma_i$
	computed in the proof, halving with each quadrupling of $N$; at generic $N$ (last row) it
	oscillates in sign with $\{t_i\}$ but stays $O(N^{-1/2})$, and adding
	$-\phi(\alpha)B_2(\{t_i\})/\sigma_i$ removes most of it. The first column converges to the
	$O(1)$ correction; the leading formula is the more accurate one while the lattice term exceeds
	half of that correction, i.e.\ for $N<625$ along this sequence, so the corrected formula
	overtakes it only from the row $N=1600$ on.
\end{remark}

% =====================================================================
\subsection{$d_{\mathbf r}$: three regimes and the law-of-large-numbers limit}
	\label{ssec:dXY-regimes}

	The standardized mismatch $a_i = N\delta_i/\sigma_i = \Theta(\sqrt N\,\delta_i)$,
	so the contribution of an edge depends on the scaling of $\delta_i$ with $N$.
	The table classifies the three \emph{canonical} scalings; it is exhaustive for
	a fixed pair $(\mathbf p,\mathbf r)$, where only the first and third rows can
	occur, while the middle row requires an $N$-dependent array
	$\mathbf r=\mathbf r_N$:
	\begin{center}
	\begin{tabular}{l|l|l|l}
		regime & $\delta_i$ & $a_i$ & edge contribution \\\hline
		\emph{matched} & $0$ & $0$ & $\sqrt{2/\pi}\,\sigma_i = \Theta(\sqrt N)$ \\
		\emph{local} & $\Theta(1/\sqrt N)$ & $\Theta(1)$ & $\sigma_i G(a_i) = \Theta(\sqrt N)$ \\
		\emph{fixed mismatch} & $\Theta(1)$ & $\Theta(\sqrt N) \to \infty$ & $N|\delta_i| = \Theta(N)$ \\
	\end{tabular}
	\end{center}
	For a fixed pair the dichotomy is clean: the leading order is $\Theta(N)$ as
	soon as any edge has a nonzero mismatch, and $\Theta(\sqrt N)$ otherwise.

\begin{remark}[Intermediate scalings]\label{rem:dXY-intermediate}
	Once $N$-dependent arrays are admitted, the table has an intermediate regime
	between its rows: for $\delta_i\asymp N^{-\beta}$ with $0<\beta<\tfrac12$ one
	has $a_i\to\infty$ while $\delta_i\to0$, and the edge contributes
	$\Theta(N^{1-\beta})$ --- strictly between $\Theta(\sqrt N)$ and $\Theta(N)$. (For
	$\beta>\tfrac12$ one has $a_i\to0$ and the edge behaves as matched to first order.)
	This regime is still deterministic to first order, but the wrong-sign part has
	to be estimated slightly more carefully than in the fixed-mismatch proof, where the
	crude bound $O(N)$ on the negative part sufficed because $|\delta_i|$ was fixed.  Put
	\(a_N=\sqrt N\,|\delta_i|/s_i\). On the wrong-sign event the shift and the fluctuation
	pull in opposite directions, so \(|\tilde D_i+N\delta_i|\le|\tilde D_i|\); bounding the
	contribution by Cauchy--Schwarz and the probability of that event by Hoeffding gives
	\[
		\E\{|\tilde D_i+N\delta_i|;\text{ wrong sign}\}
		\le\bigl(\E\tilde D_i^2\bigr)^{1/2}\,
		\Pr\{\text{wrong sign}\}^{1/2}
		\le C\sqrt N\,\exp(-c a_N^2),
	\]
	because \(\E\tilde D_i^2=O(N)\) and the one-trial edge increment is bounded.
	Dividing by \(N|\delta_i|=s_i\sqrt N\,a_N\) gives
	\(O(e^{-c a_N^2}/a_N)=o(1)\) whenever \(a_N\to\infty\). Hence
	\(\E|\tilde D_i+N\delta_i|=N|\delta_i|\,(1+o(1))\) throughout this
	intermediate range.
\end{remark}

\begin{corollary}[Law of large numbers]\label{cor:dXY-GC}
	At fixed $\mathbf p, \mathbf r$:
	\begin{equation}\label{eq:dXY-GC}
		\frac{d_{\mathbf r}}{N} \;\longrightarrow\;
		\sum_{i=1}^{n-1}|\delta_i| = \sum_{i=1}^{n-1}\bigl|(p_{i+1} - p_i) - (r_{i+1} - r_i)\bigr|,
	\end{equation}
	the $\ell_1$ distance between the gradient profiles of $\mathbf p$ and
	$\mathbf r$; the convergence holds in probability, almost surely (by the strong
	law of large numbers~\cite{billingsley}, realizing $X=X^{(N)}$ as the cell counts of the
	first $N$ trials of one i.i.d.\ categorical sequence), and in $L^1$.
	Moreover $\sum_i|\delta_i|=0$ if and only if $\mathbf p=\mathbf r$: equal
	gradient profiles make $\mathbf p-\mathbf r$ constant, and normalization forces
	the constant to be zero.
\end{corollary}

\begin{proof}
	The multinomial strong law of large numbers gives $X_i/N\to p_i$ almost surely
	for every $i$. Hence, for each fixed edge,
	\[
		\frac{D_i-N(r_{i+1}-r_i)}{N}
		=\frac{X_{i+1}}N-\frac{X_i}N-(r_{i+1}-r_i)
		\longrightarrow\delta_i \quad\text{a.s.}
	\]
	The absolute-value function is continuous, so the continuous mapping theorem
	gives convergence of the corresponding absolute value to $|\delta_i|$, almost
	surely and hence in probability.
	The number $n-1$ of edges is fixed; summing these finitely many
	convergences proves \eqref{eq:dXY-GC}. For the $L^1$ statement, note that
	$d_{\mathbf r}/N\le 2+\sum_i|r_{i+1}-r_i|$ is uniformly bounded, so no
	uniform-integrability subtlety arises (alternatively, for interior $\mathbf p$ and $n\ge3$,
	convergence in probability together with $\E(d_{\mathbf r})/N\to\sum_i|\delta_i|$ from
	Theorem~\ref{thm:E-dXY} gives $L^1$ convergence by Scheff\'e's lemma).
	The bound follows from
	\begin{align*}
		\frac1N\sum_i|X_{i+1}-X_i-N(r_{i+1}-r_i)|
		&\le \sum_i\left(\frac{X_i+X_{i+1}}N+|r_{i+1}-r_i|\right)\\
		&\le 2+\sum_i|r_{i+1}-r_i|,
	\end{align*}
	because each observation can be counted at most twice in the neighboring-edge
	sum.  Almost sure convergence plus this deterministic bound implies $L^1$
	convergence by dominated convergence.
	The identifiability statement is elementary: $\delta_i=0$ for all $i$ means
	$p_{i+1}-p_i=r_{i+1}-r_i$ for all $i$, so $p_i-r_i$ is constant in $i$, and
	$\sum_i(p_i-r_i)=0$ forces $p_i=r_i$.
\end{proof}

	For $\mathbf r=\mathbf u$ and the comb alternative $p_i=\tfrac1n+(-1)^i\varepsilon$ ($n$ even) the limit is
	$2(n-1)\varepsilon$, the separation $\E d_{\mathbf u}\approx2(n-1)\varepsilon N$ on which the
	consistency argument of~\cite{banic_elezovic_ct} rests.

\begin{corollary}[Consistency of threshold tests]\label{cor:dXY-consistency}
	Fix $n\ge2$ and a null vector $\mathbf r$. Let
	$t_N$ be any sequence with $\sqrt N\ll t_N\ll N$, and reject
	$H_0:\mathbf p=\mathbf r$ when $d_{\mathbf r}>t_N$. Then the size tends to $0$ and
	the power against every fixed $\mathbf p\ne\mathbf r$ tends to $1$.
\end{corollary}

\begin{proof}
	Under $H_0$ the statistic is nonnegative with
	$\E(d_{\mathbf r})=\E(d_{\mathbf p})\le\sum_i\sigma_i=O(\sqrt N)$ by Jensen's inequality (for interior
	$\mathbf r$ and $n\ge3$ Theorem~\ref{thm:E-TVC} gives the exact order $\Theta(\sqrt N)$), so Markov's
	inequality gives $\Pr(d_{\mathbf r}>t_N)\le\E(d_{\mathbf r})/t_N\to0$. Under a fixed
	$\mathbf p\ne\mathbf r$, Corollary~\ref{cor:dXY-GC} gives
	$d_{\mathbf r}/N\to\sum_i|\delta_i|>0$ in probability, and $t_N/N\to0$, so
	$\Pr(d_{\mathbf r}>t_N)\to1$.
	Thus the two inequalities on $t_N$ have distinct roles.  The condition
	$t_N/\sqrt N\to\infty$ pushes the threshold above the null fluctuation scale;
	the condition $t_N/N\to0$ keeps it below every fixed-alternative deterministic
	scale.  No distributional limit theorem is needed for this consistency statement;
	first moments and the law of large numbers suffice.

	The corollary describes a sequence of tests whose size tends to $0$. A test of
	\emph{fixed} level $\alpha\in(0,1)$ is consistent a fortiori: if $t_N(\alpha)$ is the
	$(1-\alpha)$-quantile of the null law of $d_{\mathbf r}$, then Markov's inequality gives
	$t_N(\alpha)\le\E_{H_0}(d_{\mathbf r})/\alpha=O(\sqrt N)=o(N)$, so the power argument above
	applies verbatim.
\end{proof}

	\medskip\noindent\emph{Uniformity.} Moving one observation to another bin changes $d_{\mathbf r}$ by at
	most $4$ (each of the two changed counts enters at most two edge terms and changes each by at
	most $1$), so McDiarmid's inequality~\cite{mcdiarmid1989}
	gives $\Pr(|d_{\mathbf r}-\E d_{\mathbf r}|\ge t)\le2e^{-t^2/(8N)}$ under every $\mathbf p$. With
	$\E_{\mathbf p}d_{\mathbf r}\ge N\sum_i|\delta_i|\ge2N\|\mathbf p-\mathbf r\|_1/n$ (Jensen, and the
	inequality $\sum_i|\delta_i|\ge2\|\mathbf p-\mathbf r\|_1/n$ noted after Definition~\ref{def:dXY}),
	the power therefore tends to $1$ uniformly over $\{\|\mathbf p-\mathbf r\|_1\ge\varepsilon\}$, and
	the size is at most $2\exp\{-(t_N-\E_{H_0}d_{\mathbf r})^2/(8N)\}$ --- a sub-Gaussian bound in
	place of Markov's.

	\paragraph{Statistical interpretation} The limit (\ref{eq:dXY-GC}) is the
	separation on which consistency rests, and
	Corollary~\ref{cor:dXY-consistency} converts it into a statement about tests:
	under every fixed alternative $d_{\mathbf r}=\Theta(N)$ while under $H_0$ it is
	$O(\sqrt N)$ (of exact order $\sqrt N$ for interior $\mathbf r$), and any threshold between the two scales works. This consistency
	rests only on the \emph{separation} of the null and alternative scales, and any statistic
	whose two scales separate is consistent in the same way --- Pearson's $\chi^2$ among
	them, with its own pair of scales ($O_P(1)$ under $H_0$, since the limit law is
	$\chi^2_{n-1}$, versus $\Theta(N)$ under a fixed alternative). Consistency is therefore
	\emph{not} by itself evidence of
	an advantage over $\chi^2$; the heuristic advantage discussed in \S\ref{sec:E-dXY}
	concerns \emph{which} alternatives are detected at moderate $N$, not comparative power.
	Critical values at finite $N$ require the null law itself: for rational $\mathbf r$ it is
	computable exactly by a weighted form of the dynamic program of~\cite[Eq.~(4)]{banic_elezovic_ct},
	listed there as future work, and for large $N$ the moments of this series feed a moment-matched
	approximation; unlike the law, the exact expectation of Propositions~\ref{prop:E-TVC-exact}
	and~\ref{prop:dXY-exact} requires no rationality of $\mathbf r$.

% =====================================================================
% =====================================================================
\subsection{$d_{\mathbf r}$: the excluded case $n=2$}
	\label{ssec:dXY-n2}

	The case $n=2$ was excluded from Theorem~\ref{thm:E-dXY} and from
	Theorem~\ref{thm:dXY-subleading} for the reason already met in
	\S\ref{sec:E-TVC-asymp}: the single edge has $\tau_1=p_1+p_2=1$, the increment
	$D_1=X_2-X_1$ moves on a span-$2$ lattice, and
	Lemma~\ref{lem:E-D1-asymp-general} does not apply. It is treated separately in
	the companion paper~\cite{elezovic_folded}, where the statistic collapses to the
	absolute deviation of a binomial about an \emph{arbitrary} point --- the
	reference mean $Nr$, which is not the mean of the data --- and where the
	exponentially small remainder of Theorem~\ref{thm:E-dXY} becomes completely
	explicit. We record the two conclusions used elsewhere in this series. With
	$p:=p_1$, $q:=1-p$, $r:=r_1$, $\bar r:=1-r$ one has $d_{\mathbf r}=2\,|X_1-Nr|$
	exactly, and for fixed $r\in(0,1)$, $r\ne p$ (after relabeling, $r<p$)
	\[
		\E(d_{\mathbf r})=2N(p-r)
		+\frac{4\,e^{-N D(r\Vert p)}\,\rho^{\,\theta_N}}{\sqrt{2\pi N\,r\bar r}}
		\Bigl[\,\frac{\theta_N}{1-\rho}+\frac{\rho}{(1-\rho)^2}\,\Bigr]
		\bigl(1+O(N^{-1})\bigr),
	\]
	with an implied constant that grows as $r\to p$, $\theta_N=\{Nr\}$, $\rho=r q/(\bar r p)<1$ and
	$D(r\Vert p)$ the Kullback--Leibler divergence
	\cite[Proposition~3.1(iii)]{elezovic_folded}. Thus at $n=2$ the $O(e^{-\gamma_iN})$ of
	Theorem~\ref{thm:E-dXY} is a Bahadur--Rao expansion~\cite{bahadur_rao1960}: the exponential rate is
	$D(r\Vert p)$, the polynomial prefactor is $N^{-1/2}$, and the constant
	oscillates boundedly with the lattice defect of the \emph{reference} mean
	rather than converging. Moreover the folded-normal profile and this saddle
	formula agree on the overlap $1\ll\sqrt N(p-r)/\sqrt{pq}\ll N^{1/6}$
	\cite[Proposition~5.1 and Remarks~5.2--5.3]{elezovic_folded}, so the matched and mismatched regimes
	form one chain --- but they must be joined through the standardized mismatch,
	because the limits $r\to p$ and $N\to\infty$ do not commute.

% =====================================================================
\subsection{$d_{\mathbf r}$: consequences of the nonzero shift}
	\label{ssec:dXY-asymmetry}

	Qualitative consequences of the nonzero mean inside the absolute values:
	\begin{enumerate}
		\item \textbf{Half-normal $\to$ folded-normal.} The per-edge expectation is
		$\sigma_i G(a_i)$, where $G$ \emph{nonlinearly interpolates} between
		the half-normal ($G(0) = \sqrt{2/\pi}$) and the deterministic ($G(a) \approx |a|$).
		\item \textbf{Asymmetric input $\mathbf p, \mathbf r$.} $\mathbf p$ enters
		through $\sigma_i$ (and $q_i$), while both vectors enter through the mismatch
		$\delta_i$. Swapping the roles changes the result, but not at leading order:
		$N\sum_i|\delta_i|$ is invariant under $\mathbf p\leftrightarrow\mathbf r$, and the asymmetry
		enters at order $\sqrt N$, through $\sigma_i$ and $G$.
		\item \textbf{Variance irrelevant under strong mismatch.} At an edge with a fixed
		mismatch $\E|\tilde D_i + N\delta_i| = N|\delta_i| + O(e^{-\gamma_iN})$; $\sigma_i$, the
		fluctuation scale under $\mathbf p$, disappears from the leading order.
		\item \textbf{Comb-sign symmetry survives.} For $\mathbf p_\pm=\mathbf r\pm\mathbf e$
		($\sum_ie_i=0$) the mismatches are $\pm(e_{i+1}-e_i)$, so the leading term
		$N\sum_i|\delta_i|$ is the same for both signs; the two signs differ only through
		$\sigma_i(\mathbf p_\pm)$, at lower order. The uncentered $d_{\mathbf u}$ separates them by
		$O(N)$ (paragraph \emph{The uncentered statistic} after the proof of Theorem~\ref{thm:E-dXY}).
		\item \textbf{Edge-dependent regimes.} Under $\mathbf p\ne\mathbf r$ the edges can sit in
		different regimes (Corollary~\ref{cor:dXY-mixed}), so no single scaling describes all
		summands.
	\end{enumerate}

% =====================================================================
\subsection{$d_{\mathbf r}$: numerical validation}
	\label{ssec:dXY-numerics}

	The folded-normal formula~\eqref{eq:E-dXY} was checked against the exact finite-$N$ value of
	Proposition~\ref{prop:dXY-exact} on four
	$n=8$ laws: the \emph{comb} $p_k=\tfrac1n+\tfrac{(-1)^k}{16}$, the \emph{strong
	comb} $p_k=\tfrac1n+\tfrac{(-1)^k}{10}$, a \emph{shallower comb}
	$r_k=\tfrac1n+\tfrac{(-1)^k}{32}$, and the \emph{geometric} law
	$p_k\propto(0.7)^{\,k-1}$ (all for $k=1,\dots,n$). The results are collected in
	Table~\ref{tab:dXY-MC}.
	\begin{table}[ht]
	\centering
	\begin{tabular}{l|r|r|r|c}
		case $(\mathbf p, \mathbf r)$ & $N$ & exact (Prop.~\ref{prop:dXY-exact}) & pred (\ref{eq:E-dXY}) & rel.\ error \\\hline
		strong-comb vs uniform & 200  & $280.000$  & $280.000$  & $<10^{-6}\%$ \\
		& 5000 & $7000.000$ & $7000.000$ & $<10^{-6}\%$ \\
		comb vs same ($=d_{\mathbf p}$) & 200  & $38.165$   & $38.239$   & $+0.194\%$ \\
		& 5000 & $191.181$  & $191.196$  & $+0.0077\%$ \\
		comb vs shallower & 200  & $88.743$   & $88.782$   & $+0.044\%$ \\
		& 5000 & $2187.500$ & $2187.500$ & $<10^{-6}\%$ \\
		geometric vs uniform & 1000 & $293.321$  & $293.346$  & $+0.0085\%$ \\
	\end{tabular}
	\caption{Folded-normal prediction \eqref{eq:E-dXY} against the exact finite-$N$ expectation of
	Proposition~\ref{prop:dXY-exact} ($n=8$; the four laws are defined in the text above); the relative
	error is $(\text{pred}-\text{exact})/\text{exact}$. A Monte Carlo run ($K=200\,000$ replications,
	seed $20\,261\,008$) agrees with the exact values within $2.1$ standard errors in every row.}
	\label{tab:dXY-MC}
	\end{table}
	The leading formula agrees with the exact expectation to within $0.2\%$ for $N\ge200$ in all
	regimes. On edges with a fixed mismatch the error is exponentially small (zero to the digits
	shown); the largest error is in the matched ($d_{\mathbf p}$) case, where it decreases as $O(1/N)$
	(from $0.194\%$ at $N=200$ to $0.0077\%$ at $N=5000$) --- the finite-$N$ corrections of
	Theorem~\ref{thm:E-TVC} reproduce the matched entries to six digits ($38.16526$ at $N=200$).
	In the rows ``comb vs shallower'' ($N=200$) and ``geometric vs uniform'' the standardized mismatches
	are moderate ($|a_i|\le4.2$), and the $O(1)$ correction of Corollary~\ref{cor:dXY-mixed} reduces the
	remaining absolute error by a further factor $4$--$10$ ($0.039\to-0.009$ and $0.025\to0.003$).

% =====================================================================
\section{Conclusion}
	\label{sec:conclusion}

	This paper extends the expectation theory of discrete total variation from the
	uniform null to a general multinomial law. Its exact part is the finite mixture of
	Proposition~\ref{prop:E-TVC-exact} (and Proposition~\ref{prop:dXY-exact} under a different
	sampling law): conditioning on the total of each adjacent pair reduces every edge to a
	translated binomial absolute deviation, for every $n$, every $N$ and every $\mathbf p$. Its
	asymptotic part, for the matched statistic $d_{\mathbf p}$, is the three-term expansion of
	Theorem~\ref{thm:E-TVC}, proved by a weighted lattice Edgeworth estimate and a corner
	Euler--Maclaurin summation: the shape-dependent coefficients, encoded by the adjacent
	probabilities, separate from the universal Edgeworth and Bernoulli-polynomial coefficients, the
	lattice part oscillates with $\{N(p_{j+1}-p_j)\}$, and all half-integer orders vanish by parity.

	The same per-edge viewpoint gives the leading behavior of $d_{\mathbf r}$ under a
	possibly different sampling law. In the matched case one recovers the half-normal mean
	$\sqrt{2/\pi}\,\sigma_i$ per edge, while a fixed mismatch gives a deterministic leading term
	with only exponentially small corrections; the local-alternative expansion of
	Theorem~\ref{thm:dXY-subleading} records the transition between these two regimes, with its
	$O(1)$ skewness term. Fixed alternatives give the law-of-large-numbers limit
	$d_{\mathbf r}/N\to\sum_i|\delta_i|$, the $\ell_1$ distance between adjacent-gradient profiles,
	which vanishes only at $\mathbf p=\mathbf r$ and makes threshold tests consistent. The
	two-bin case, where the edge walk has span $2$, is recorded from the companion
	papers~\cite{elezovic_mad,elezovic_folded}: there the exponentially small correction is
	explicit (a Bahadur--Rao form) and the limits $r\to p$ and $N\to\infty$ do not commute.

	Together with the uniform case, these results provide the expectation part of
	the general theory. The corresponding variance theory is the subject of
	Part~4~\cite{banic_elezovic_p4}; the formulas proved here supply the centering and the
	per-edge asymptotic input needed there.

% =====================================================================
% Bibliography.
% =====================================================================
\section*{Acknowledgments}
	The authors used an AI assistant (Anthropic's Claude) for literature search,
	symbolic and numerical verification, and editorial support; all mathematical
	results were checked independently by the authors.

\end{document}